\documentclass[12pt]{amsart}

\usepackage{color}
\usepackage{comment}
\usepackage{lineno}
\usepackage{ulem}

\def\R{\mathbb{R}}
\def\RN{\mathbb{R}^N}

\def\calL{\mathcal{L}}

\def\e{\varepsilon}

\def\ul{{ ul}}
\def\loc{{loc}}

\def\bw{\bar{w}}

\def\XXint#1#2#3{{\setbox0=\hbox{$#1{#2#3}{\int}$}
\vcenter{\hbox{$#2#3$}}\kern-.5\wd0}}

\newcommand{\di}{\mathop{}\!\mathrm{d}}

\DeclareMathOperator{\esssup}{ess\,sup}
\DeclareMathOperator{\supp}{supp}

\theoremstyle{plain}
	\newtheorem{theorem}{Theorem}[section]
	\newtheorem{lemma}[theorem]{Lemma}
	\newtheorem{corollary}[theorem]{Corollary}
	\newtheorem{definition}[theorem]{Definition}
	\newtheorem{proposition}[theorem]{Proposition}
	\newtheorem{remark}[theorem]{Remark}

\theoremstyle{plain}

\makeatletter
 \@addtoreset{equation}{section}
\makeatother

\begin{document}
\title[Semilinear heat equations  with exponential growth]{
Threshold property of a singular stationary solution for
semilinear heat equations with exponential growth
}

\author{Kotaro Hisa}
\thanks{}
\address[Kotaro Hisa]{Graduate School of Mathematical Sciences, The University of Tokyo,
 3-8-1 Komaba, Meguro-ku, Tokyo 153-8914, Japan,
{\rm{\texttt{hisak@ms.u-tokyo.ac.jp}}}
}

\author{Yasuhito Miyamoto}
\thanks{ORCiD of YM is 0000-0002-7766-1849}
\thanks{YM was supported by JSPS KAKENHI Grant Number 24K00530.}
\address[Yasuhito Miyamoto]{Graduate School of Mathematical Sciences, The University of Tokyo,
 3-8-1 Komaba, Meguro-ku, Tokyo 153-8914, Japan,
{\rm{\texttt{miyamoto@ms.u-tokyo.ac.jp}}}
}

\begin{abstract}
Let $N\ge 3$.
We are concerned with a Cauchy problem of the semilinear heat equation
\[
\begin{cases}
\partial_tu-\Delta u=f(u), & x\in\RN,\ t>0,\\
u(x,0)=u_0(x), & x\in\RN,
\end{cases}
\]
where $f(0)=0$, $f$ is nonnegative, increasing and convex, $\log f(u)$ is convex for large $u>0$ and some additional assumptions are assumed.
We establish a positive radial singular stationary solution $u^*$ such that $u^*(x)\to\infty$ as $|x|\to 0$.
Then, we prove the following:
The problem has a nonnegative global-in-time solution if $0\le u_0\le u^*$ and $u_0\not\equiv u^*$, while the problem has no nonnegative local-in-time solutions $u$ such that $u\ge u^*$ if $u_0\ge u^*$ and $u_0\not\equiv u^*$.
\end{abstract}

\date{\today}
\maketitle

%\linenumbers

\bigskip

%%%%%%%%%%%%%%%%%%%%%%%%%%%%%%%%%%%%%%%%%%%%%%%%%%%%%%%
%%%%%%%%%%%%%%%%%%%%%%%%%%%%%%%%%%%%%%%%%%%%%%%%%%%%%%%
%%%%%%%%%%%%%%%%%%%%%%%%%%%%%%%%%%%%%%%%%%%%%%%%%%%%%%%
% Section 1
%%%%%%%%%%%%%%%%%%%%%%%%%%%%%%%%%%%%%%%%%%%%%%%%%%%%%%%
%%%%%%%%%%%%%%%%%%%%%%%%%%%%%%%%%%%%%%%%%%%%%%%%%%%%%%%
%%%%%%%%%%%%%%%%%%%%%%%%%%%%%%%%%%%%%%%%%%%%%%%%%%%%%%%
%\newpage

\section{\bf Introduction and main results}
This paper is concerned with the Cauchy problem for the  semilinear heat equation
\begin{equation}\label{S1E1}
	\left\{
	\begin{aligned}
		&\partial_t u  -\Delta  u =  f(u),	\quad && x\in \mathbb{R}^N,\,\, t>0,\\
		&u(x,0) =  u_0(x),		\quad && x\in \mathbb{R}^N,
	\end{aligned}
	\right.
\end{equation}
where  $N\geq3$ and $u_0$ is a nonnegative measurable function on $\mathbb{R}^N$.
When $f(u)=u^p$, problem~\eqref{S1E1} has an explicit positive singular stationary solution and \cite{GV97} showed the following: Under certain condition on $p>1$, problem~\eqref{S1E1} has a solution for initial data less than the singular solution, while a solution of problem~\eqref{S1E1} that is larger than the singular stationary solution does not exist.
When $f(u)=e^u$, the Cauchy-Dirichlet problem on the unit ball has an explicit positive singular stationary solution and \cite{PV95} proved this phenomenon.
Our aim is to construct a positive singular stationary solution of problem~\eqref{S1E1} and to prove this phenomenon when $f$ has an exponential or superexponential growth and the domain is $\R^N$.

Let $p_S$ denote the critical Sobolev exponent, {\it i.e.},
\[
p_S :=\frac{N+2}{N-2}.
\]	
Throughout this paper, we make the following assumptions on $f$:
\begin{itemize}
\item[(A1)]  $f\in C^1[0,\infty)\cap C^2(0,\infty)$, $f(0)=0$, and $f'(0)=0$;
\item[(A2)]  $f'(u)>0$ for $u>0$ and $f''(u)>0$ for $u>0$;
\item[(A3)] $g(u):=\log f(u)$ is convex for large $u>0$ and
\[
\lim_{u\to\infty}\frac{g''(u)}{g'(u)^2}=0;
\]
\item[(A4)] The following holds:
\begin{equation*}
Q(u):=uf(u) - (p_S+1) \int_{0}^u f(s) \, \di s \ge 0\quad\mbox{for}\ u\ge 0.
\end{equation*}
\end{itemize}
We define $f(u)=0$ for $u<0$, and this will be used in the proof of Theorem~\ref{S1T1} below.

For $x\in\mathbb{R}^N$ and $\rho>0$, set $B(x,\rho):=\left\{ y\in\RN;\ |x-y|<\rho\right\}$.
We define a uniformly local $L^p$ space as follows:
For $1\le p\le\infty$, 
\[
L^p_{\ul}(\RN):=\left\{ u\in L^p_{\loc}(\RN);\ \left\|u\right\|_{L^p_{ul}(\RN)}<\infty\right\},
\]
where
\[
\left\|u\right\|_{L^p_{\ul}(\RN)}:=
\begin{cases}
\displaystyle{\sup_{z\in\RN}\left(\int_{B(z,1)}|u(x)|^p \, \di x\right)^\frac{1}{p}} & \textrm{if}\ 1\le p<\infty,\\
\displaystyle{\esssup\displaylimits_{z\in\RN}\left\|u\right\|_{L^{\infty}(B(z,1))}} & \textrm{if}\ p=\infty.
\end{cases}
\]
It is obvious that $L^{\infty}_{\ul}(\RN)=L^{\infty}(\RN)$ and that $L^{p_1}_{\ul}(\RN)\subset L^{p_2}_{\ul}(\RN)$ if $1\le p_2\le p_1 <\infty$.
Let $\calL^p_{\ul}(\RN)$ denote the closure of the space of bounded uniformly continuous functions $BUC(\RN)$ in the space $L^p_{\ul}(\RN)$, {\it i.e.},
\begin{equation}\label{S3E0}
\calL^p_{\ul}(\RN):=\overline{BUC(\RN)}^{\left\|\,\cdot\,\right\|_{L^p_{\ul}(\RN)}}.
\end{equation}

Let $G= G(x,t)$ be the fundamental solution of 
\[
\partial_t v -\Delta v = 0 \quad \mbox{in} \quad \mathbb{R}^N\times (0,\infty),
\]
that is, 
\[
G(x,t) := \frac{1}{(4\pi t)^\frac{N}{2}} \exp \left(-\frac{|x|^2}{4t}\right)
\]
for $x\in \mathbb{R}^N$ and $t>0$.
For $u_0\in L^1_{\ul} (\mathbb{R}^N)$, let $S(t)u_0$ be
\[
[S(t)u_0](x) := \int_{\mathbb{R}^N} G(x-y,t) u_0(y) \, \di y
\]
for $x\in \mathbb{R}^N$ and $t>0$.

\begin{definition}\label{S1D1}
Let $u$ be a nonnegative measurable function on $\RN\times[0,T)$, where $T\in(0,\infty]$.
\begin{itemize}
\item[(i)] We say that $u$ is a solution of problem \eqref{S1E1} in $\mathbb{R}^N\times [0,T)$ if
\begin{equation}\label{S1D1E2}
\begin{split}
\infty>u(x,t)
&=\int_{\mathbb{R}^N}G(x-y, t )u_0(y) \, \di y+\int_0^t\int_{\mathbb{R}^N}G(x-y,t-s)f(u(y,s)) \, \di y\di s
\end{split}
\end{equation}
for a.a.~$(x,t)\in\R^N\times [0,T)$.
\item[(ii)] If  $u$ satisfies \eqref{S1D1E2} with $=$ replaced by $\ge$, then we say that $u$ is a supersolution in $\mathbb{R}^N\times [0,T)$.
\end{itemize}
\end{definition}

First, we construct a positive radial singular solution.
\begin{proposition}\label{S1P1}
Suppose that $N\ge 3$ and that $f$ satisfies {\rm (A1)}--{\rm (A4)}.
The problem
\begin{equation}\label{SS}
\begin{cases}
\Delta u+f(u)=0, & x\in\R^N\setminus\{0\},\\
u(x)\ge 0, & x\in\R^N\setminus\{0\},\\
u(x)\ \mbox{is radial},\\
u(x)\to\infty\quad\mbox{as}\quad |x|\to 0,
\end{cases}
\end{equation}
has a unique solution $u^*\in C^2(\R^N\setminus\{0\})$, which is called a singular solution.
Moreover, $u^*\in \calL^1_{\ul}(\R^N)$ and $u^*$ satisfies
\[
u^*(x)=F^{-1}\left[\frac{|x|^2}{2N-4}(1+o(1))\right]\quad\mbox{as}\quad |x|\to 0,
\]
where $F^{-1}$ is the inverse function of $F$ defined by
\begin{equation}\label{F}
F(u)=\int_u^{\infty}\frac{ds}{f(s)} \quad\mbox{for}\quad u>0.
\end{equation}
\end{proposition}

We are in a position to state the main theorem on the solvability of problem~\eqref{S1E1}.
\begin{theorem}\label{S1T1}
Suppose that $N\ge 3$ and that $f$ satisfies {\rm (A1)}--{\rm (A4)}.
Let $u^*$ be the singular solution of problem~\eqref{SS} given in Proposition~{\rm \ref{S1P1}}.
Then the following hold:
\begin{itemize}
\item[(i)] If $u_0\in L^1_{\ul}(\R^N)$, $u_0 \le u^*$ a.e.~in $\R^N$, and $u_0 \not\equiv u^*$ a.e. in $\mathbb{R}^N$ (and hence $u_0<u^*$ on a set of positive measure), then problem~\eqref{S1E1} has a solution in $\mathbb{R}^N\times[0,\infty)$ in the sense of Definition~{\rm \ref{S1D1}}.
In particular, this solution satisfies the equation in the classical sense for all $(x,t)\in \RN\times (0,\infty)$;
\item[(ii)] If $u_0 \equiv u^*$ a.e.~in $\mathbb{R}^N$, then $u_0$ is a singular stationary solution of problem~\eqref{S1E1} in the sense of Definition~{\rm \ref{S1D1}};
\item[(iii)] If $u_0\in L^1_{\ul}(\R^N)$, $u_0 \ge u^*$ a.e.~in $\R^N$, and $u_0 \not\equiv u^*$ a.e.~in $\mathbb{R}^N$ (and hence $u^*<u_0$ on a set of positive measure), then problem~\eqref{S1E1} has no local-in-time  solutions such that $u(x,t)\ge u^*(x)$ for a.a.~$(x,t)\in\RN\times [0,T)$ in the sense of Definition~{\rm \ref{S1D1}} for any $T>0$. 
\end{itemize}
\end{theorem}
Theorem~\ref{S1T1} says that the singular solution $u^*$ of problem~\eqref{SS} is a threshold of the existence and nonexistence of nonnegative solutions of problem~\eqref{S1E1}.
The assumption (A3) is technical.
We expect that Theorem~\ref{S1T1} holds when $f$ satisfies a certain growth condition and the limit $\lim_{u\to\infty}f'(u)^2/f(u)f''(u)$ exists.

\begin{remark}
In this paper we do not treat with the uniqueness of the solution.
The situation is delicate.
If $f(u)=u^p$ and $u_0$ is the explicit singular stationary solution, then it was shown in \cite[Theorem 10.1]{GV97} that \eqref{S1E1} has at least two positive solutions for a certain range of $p>1$.
See also \cite{MT03,NS85,SW03,T02,T21} for the non-uniqueness in other pure power cases.
If $f(u)=e^u$, $u_0$ is the explicit singular stationary solution, and the domain is a unit ball, then it was shown in \cite[Theorem 16.1]{GV97} that the Cauchy-Dirichlet problem has at least two positive solutions for $3\le N\le 9$.
In the case of our class of nonlinear terms the uniqueness of a solution is an open problem if the initial data is singular.
\end{remark}
\bigskip

Let us recall related results.
Weissler~\cite{W80} started studying the solvability of problem~\eqref{S1E1} with $f(u)=|u|^{p-1}u$ when $u_0\not\in L^{\infty}(\R^N)$.
In \cite{W80} the following was proved:
If $u_0\in L^q(\R^N)$ for either $q=N(p-1)/2>1$ or $q>N(p-1)/2$ and $q\ge 1$, then problem~\eqref{S1E1} has a local-in-time solution, while there exists a nonnegative initial function $u_0\in L^q(\R^N)$ for some $1\le q<N(p-1)/2$ such that problem~\eqref{S1E1} has no nonnegative local-in-time solutions.
When $p>N/(N-2)$, problem~\eqref{S1E1} with $f(u)=|u|^{p-1}u$ has a singular stationary solution
\[
u^*(x)=L|x|^{-\frac{2}{p-1}},\ \ L:=\left\{\frac{2}{p-1}\left(N-2-\frac{2}{p-1}\right)\right\}^{1/(p-1)}.
\]
We easily see that, for small $\e>0$,
\begin{equation}\label{S1E4}
\int_{B(0,\e)}u^*(x)^q\di x=C_N\int_0^{\e}r^{-\frac{2q}{p-1}+N-1}\di r
\begin{cases}
=\infty & \textrm{for}\ q=\frac{N(p-1)}{2},\\
<\infty & \textrm{for}\ 1\le q<\frac{N(p-1)}{2}.
\end{cases}
\end{equation}
Roughly speaking, \eqref{S1E4} indicates that $u^*$ is a threshold that separates the existence and nonexistence of local-in-time solutions.

Now, it is believed that a singularity of a singular stationary solution gives an optimal singularity in a wide class of semilinear parabolic equations.
To be more precise, if a nonnegative initial function $u_0$ has a singularity weaker (resp.~stronger) than a singular stationary solution $u^*$, then problem~\eqref{S1E1} has a nonnegative local-in-time solution (resp.~no nonnegative local-in-time solutions).
This type of existence and nonexistence theorems can be found in \cite{BP85,BC96, FHIL23,FHIL24,FI18,HI18,IKO20,IRT19,LRSV16,M21,MS24,W80,W86}.

Next, let us study our problem~\eqref{S1E1} by using \cite{FI18,FHIL24}.
Applying \cite[Theorems~1.1 and 1.2]{FI18} to our problem, we see the following:
Assume that (A1)--(A4) hold. If $u_0\ge 0$ and
\begin{equation}\label{S1E5}
\frac{1}{F(u_0)^{N/2}}\in \calL^1_{ul}(\R^N),
\end{equation}
then problem~\eqref{S1E1} has a nonnegative local-in-time solution which is classical on $\R^N\times (0,T)$ for some $T>0$, while there exists a nonnegative measurable function $u_0$ such that
\begin{equation}\label{S1E6}
\frac{1}{F(u_0)^q}\in L^1_{ul}(\R^N)\ \ \textrm{for some}\ \ 0<q<\frac{N}{2}
\end{equation}
and that problem~\eqref{S1E1} has no nonnegative local-in-time solutions.
Here, we consider the case $u_0=u^*$. By Proposition~\ref{S1P1} we see that
\[
\frac{1}{F(u^*)}=\frac{2N-4}{|x|^2}(1+o(1))\ \ \textrm{as}\ \ |x|\to 0.
\]
Then, for small $\e>0$, 
\begin{equation}\label{S1E7}
\int_{B(0,\e)}\frac{\di x}{F(u^*)^q}\sim C_N\int_0^{\e}r^{-2q+N-1}{\di r}
\begin{cases}
=\infty & \textrm{for}\ q=\frac{N}{2},\\
<\infty & \textrm{for}\ 0<q<\frac{N}{2}.
\end{cases}
\end{equation}
Therefore, $1/F(u^*)^{N/2}\not\in \calL^1_{ul}(\R^N)$.
Since (\ref{S1E5}) is not satisfied, \cite[Theorem~1.1]{FI18} does not guarantee the existence of a nonnegative local-in-time solution.
Moreover, \eqref{S1E7} indicates that $u^*$ is a threshold that separates the existence \eqref{S1E5} and the nonexistence \eqref{S1E6}.
Recently, \cite{FHIL24} proved the existence of a nonnegative local-in-time solution when $1/F(u_0)\le\e |x|^{-2}$ for a.a.~$x\in\mathbb{R}^N$ and for sufficiently small $\e>0$,
while they also proved the nonexistence of nonnegative local-in-time solutions 
when $1/F(u_0)\ge\gamma |x|^{-2}$ for a.a.~$x$ in a neighbourhood of the origin and for sufficiently large $\gamma>0$ (see \cite[Theorem~1.1]{FHIL24}).
By Proposition~\ref{S1P1} this implies that when $u_0$ has the same singularity at the origin as $u^*$, a nonnegative local-in-time solution may or may not exist depending on the magnitude of a coefficient.
When $u_0=u^*$, previous studies including \cite{W80,FHIL24,FI18} cannot guarantee the existence of a nonnegative local-in-time solution of problem~\eqref{S1E1}, though $u^*(x)$ itself is a stationary solution to problem~\eqref{S1E1}.

Let us go back to our results.
Using Theorem~\ref{S1T1}, we see the following:
If $u_0$ consists of $u^*$ plus any small positive bump, then no nonnegative local-in-time solutions satisfying $u(x,t)\ge u^*(x)$ exist even if this positive bump is located far away from the singular point, which is the origin.
On the other hand, if $u_0$ consists of $u^*$ plus any small negative bump, then a nonnegative global-in-time solution $u$ exists and $u$ immediately becomes bounded for every $t>0$.
Therefore, not only the strength of a singularity but also a whole shape of $u_0$ is important for the existence and nonexistence of  solutions.
Indeed, $u^*$ is a threshold.
As mentioned in Section 1, this type of theorems were proved for the Cauchy problem for $\partial_tu-\Delta u=u^p$ in \cite{GV97} and for the Cauchy-Dirichlet problem for $\partial_tu-\Delta u=2(N-2)e^u$ on a unit ball in \cite{PV95}.
Proofs in \cite{GV97,PV95} use specialities of $u^p$ and $e^u$.
The aim of the present paper is to prove Theorem~\ref{S1T1} for a rather wide class of nonlinearities.
\medskip

Let us mention a strategy of the proof.
First, we show that $u^*$ constructed in Proposition~\ref{S1P1} is a stationary singular solution in the sense of Definition~\ref{S1D1}.
In the proof of Theorem~\ref{S1T1}~(i) we use a monotone method.
Specifically, we construct the maximal solution.
Since $u_0\le u^*$, a supersolution of problem~\eqref{S1E1} always exists.
It is enough to show that the maximal solution is bounded for every $t>0$.
We prove the boundedness by checking the integrability condition \eqref{S5L1E4-}.
The nonexistence of nonnegative local-in-time solutions is the main issue in the present paper.
We use a notion of weak solutions defined by Definition~\ref{S6D1}.
When the assumptions of Theorem~\ref{S1T1}~(iii) are satisfied, we show that \eqref{S1E1} has no nonnegative weak solutions, using a method developed in \cite{PV95} in the framework of weak solutions.
In Appendix it is proved that a solution in the sense of Definition~\ref{S1D1} is also a weak solution.
Hence, \eqref{S1E1} has no nonnegative local-in-time solutions in the sense of Definition~\ref{S1D1}.
\medskip

This paper consists of seven sections.
In Section~2 we give two examples for Theorem~\ref{S1T1}.
In Section~3 we construct a positive radial singular solution $u^*$ of problem~\eqref{SS} and prove Proposition~\ref{S1P1}.
We also collect useful lemmas.
In Section~4 we show that $u^*$ given in Proposition~\ref{S1P1} is a stationary solution of problem~\eqref{S1E1} in the sense of Definition~\ref{S1D1}.
Then, Theorem~\ref{S1T1}~(ii) is proved.
In Section~5 we construct a solution if $0\le u_0\le u^*$ and $u_0\not\equiv u^*$.
Then, Theorem~\ref{S1T1}~(i) is established.
In Section~6 we prove Theorem~\ref{S1T1}~(ii), using weak solutions defined by Definition~\ref{S6D1}.
Section 7 is an Appendix.
We show that if $u$ is a solution in the sense of Definition~\ref{S1D1}, then $u$ is a weak solution.

%%%%%%%%%%%%%%%%%%%%%%%%%%%%%%%%%%%%%%%%%%%%%%%%%%%%%%%%
%%%%%%%%%%%%%%%%%%%%%%%%%%%%%%%%%%%%%%%%%%%%%%%%%%%%%%%%
%%%%%%%%%%%%%%%%%%%%%%%%%%%%%%%%%%%%%%%%%%%%%%%%%%%%%%%%
% Section 2
%%%%%%%%%%%%%%%%%%%%%%%%%%%%%%%%%%%%%%%%%%%%%%%%%%%%%%%%
%%%%%%%%%%%%%%%%%%%%%%%%%%%%%%%%%%%%%%%%%%%%%%%%%%%%%%%%
%%%%%%%%%%%%%%%%%%%%%%%%%%%%%%%%%%%%%%%%%%%%%%%%%%%%%%%%
%\newpage
\section{\bf Two examples of nonlinearities $f$}
Roughly speaking, we will construct a nonnegative, increasing, and convex function satisfying:
\[
f(u)\sim
\begin{cases}
u^p & \textrm{for small $u>0$,}\\
e^{g(u)} & \textrm{for large $u>0$},
\end{cases}
\]
where $p\ge p_S$ and $g(u)$ is a convex function.
Moreover, (A4) has to be satisfied.
\subsection{Example 1}
The first example is 
\[
f(u)=u^pe^{u^q},\ \ p\ge p_S,\ \ q>1.
\]
It is obvious that (A1) and (A2) hold.
We check (A3) and (A4).
Since $g(u)=u^q+p\log u$, we have
\[
g''(u)=\frac{q(q-1)u^q-p}{u^2}>0\ \ \mbox{for large}\ \ u>0
\]
and
\[
\frac{g''(u)}{g'(u)^2}=\frac{q(q-1)u^q-p}{(qu^q+p)^2}\to 0\ \ \mbox{as}\ \ u\to\infty.
\]
Hence, (A3) holds.
Since $Q(0)=0$ and
\[
Q'(u)=(p+1)u^pe^{u^q}+qu^{p+q}e^{u^q}-(p_S+1)u^pe^{u^q}\ge 0\ \ \mbox{for}\ \ u\ge 0,
\]
we see that $Q(u)\ge 0$ for $u\ge 0$.
Thus, (A4) holds.

\subsection{Example 2}
The second example is
\[
f(u)=\chi(u)e^{au},
\]
where $a=20$, $\chi(u):=\int_0^u\chi'(s)ds$ and
\[
\chi'(u)=
\begin{cases}
0 & \mbox{if}\ u\ge 4,\\
5(u-4)^4 & \mbox{if}\ 3\le u\le 4,\\
10-5(u-2)^4 & \mbox{if}\ 1\le u\le 3,\\
5u^4 & \mbox{if}\ 0\le u\le 1.
\end{cases}
\]
In particular, $f\in C^1[0,\infty)\cap C^2(0,\infty)$ and
\[
f(u)=
\begin{cases}
u^5e^{au} & \mbox{for}\ 0\le u\le 1,\\
20e^{au} & \mbox{for}\ u\ge 4.
\end{cases}
\]
It is obvious that (A1) holds.
We check (A2)--(A4).
Since $|\chi'|\le 10$ and $|\chi''|\le 20$, by elementary calculation we can see that
\[
f''(u)=A^2\left(\chi+\frac{2\chi'}{a}+\frac{\chi''}{a^2}\right)e^{Au}\ge 0\ \ \mbox{for}\ \ u\ge 0.
\]
Then, (A2) holds.
Since $g(u)=au+\log 20$ for $u\ge 4$, we see that $g''(u)\ge 0$ for $u\ge 4$ and
$g''(u)/g'(u)^2=0$ for $u\ge 4$.
Hence, (A3) holds.
By elementary calculation we can check that $Q(0)=0$ and
\[
Q'(u)=\left(u\chi'+au\chi-p_S\chi\right)e^{au}\ge 0\ \ \mbox{for}\ \ u\ge 0.
\]
Then, $Q(u)\ge 0$ for $u\ge 0$.
Thus, (A4) holds.

%%%%%%%%%%%%%%%%%%%%%%%%%%%%%%%%%%%%%%%%%%%%%%%%%%%%%%%%
%%%%%%%%%%%%%%%%%%%%%%%%%%%%%%%%%%%%%%%%%%%%%%%%%%%%%%%%
%%%%%%%%%%%%%%%%%%%%%%%%%%%%%%%%%%%%%%%%%%%%%%%%%%%%%%%%
% Section 3
%%%%%%%%%%%%%%%%%%%%%%%%%%%%%%%%%%%%%%%%%%%%%%%%%%%%%%%%
%%%%%%%%%%%%%%%%%%%%%%%%%%%%%%%%%%%%%%%%%%%%%%%%%%%%%%%%
%%%%%%%%%%%%%%%%%%%%%%%%%%%%%%%%%%%%%%%%%%%%%%%%%%%%%%%%
\section{\bf Existence of a singular solution of problem~\eqref{SS}}
\begin{lemma}\label{S3L2}
Suppose that {\rm (A1)}--{\rm (A3)}.
Then,
\[
\lim_{u\to\infty}f'(u)F(u)=1.
\]
\end{lemma}
\begin{proof}
By (A3) we have
\[
\lim_{u\to\infty}\frac{f'(u)^2}{f(u)f''(u)}=\lim_{u\to\infty}\frac{1}{1+g''(u)/g'(u)^2}=1.
\]
Then, if $u>0$ is large, we have $f'(u)^2/f(u)f''(u)\le 2$, and hence
\[
\frac{f'(u)}{f(u)}\le 2\frac{f''(u)}{f'(u)}
\]
for large $u>0$.
Solving the differential inequality, we have $f(u)\le Cf'(u)^2$ for large $u>0$.
It follows from (A1) and (A2) that $f(u)\to \infty$ as $u\to\infty$.
Thus, $f'(u)\to\infty$ as $u\to\infty$.
Then, By L'Hospital's rule we have
\[
\lim_{u\to\infty}\frac{F(u)}{1/f'(u)}
=\lim_{u\to\infty}\frac{-1/f(u)}{-f''(u)/f'(u)^2}
=\lim_{u\to\infty}\frac{f'(u)^2}{f(u)f''(u)}.
\]
Since $f(u)=e^{g(u)}$, by (A3) we have
\[
\lim_{u\to\infty}\frac{f'(u)^2}{f(u)f''(u)}
=\lim_{u\to\infty}\frac{1}{1+g''(u)/g'(u)^2}=1.
\]
Thus, the proof is complete.
\end{proof}

Let us consider the equation
\begin{equation}\label{ODE}
u''+\frac{N-1}{r}u'+f(u)=0.
\end{equation}
As mentioned in Proposition~\ref{S1P1},  we call a solution $u^*$ of equation \eqref{ODE} a singular solution if $u^*(r)\to\infty$ as $r\to 0^+$.

For $\alpha>0$, let $u(r,\alpha)$ denote a solution of the initial value problem~\eqref{ODE} with $u(0,\alpha)=\alpha$ and $u_r(0,\alpha)=0$.
Because of Lemma~\ref{S3L2}, all the assumptions of \cite[Theorem~1.1 and Lemma~2.5]{MN23} follow from (A1)--(A3) and we will use the following version of \cite[Theorem~1.1 and Lemma~2.5]{MN23}:
\begin{proposition}\label{S2P1}
Suppose that $N\ge 3$ and that {\rm (A1)}--{\rm (A3)} hold.
Then, there exists a unique positive singular solution $u^*$ of equation \eqref{ODE} for $0<r\le r_1$ with some $r_1>0$, and the regular solution $u(r,\alpha)$ satisfies
\[
u(r,\alpha)\to u^*(r)\ \ \textrm{in}\ \ C^2_{loc}(0,r_1]\ \ \textrm{as}\ \ \alpha\to\infty.
\]
Furthermore, the positive singular solution $u^*$ satisfies
\begin{equation}\label{S2P1E2}
-r^{N-1}(u^*)'(r)=\int_0^rf(u^*(s))s^{N-1} \, \di s
\end{equation}
for $0<r<r_1$ and
\begin{equation}\label{S2P1E3}
u^*(r)=F^{-1}\left[\frac{r^2}{2N-4}(1+o(1))\right]\ \ \textrm{as}\ \ r\to 0,
\end{equation}
where $F^{-1}$ is the inverse function of $F(u)$ defined by \eqref{F}.
\end{proposition}

\begin{lemma}\label{S3L3}
Suppose that $N\ge 3 $ and that {\rm (A1)}--{\rm (A3)} hold.
Let $u^*$ be the unique positive singular solution of equation \eqref{ODE} on $0<r<r_1$ given in Proposition~{\rm \ref{S2P1}}.
For any small $\delta>0$, there exist $R_0\in (0,r_1)$ and $C>0$ such that
\[
u^*(r) \le C r^{-2\delta}  \quad \mbox{and} \quad
|(u^*)'(r)| \le C r^{-1-2\delta} 
\]
for $r\in (0,R_0)$.
\end{lemma}

\begin{proof} 
Let $\delta>0$ be sufficiently small.
By Lemma~\ref{S3L2} we have
\[
\lim_{u\to\infty} f'(u) F(u)=1.
\]
Then there exists $M>0$ such that $f'(u) F(u) \le 1+\delta$ for all $u\ge M$.
This implies that
\[
\frac{f'(u)}{f(u)} \le \frac{1+\delta}{f(u) F(u)} = -\frac{(1+\delta) F'(u)}{F(u)}
\]
for all $u\ge M$.
Integrating the above inequality with respect to $u$ over $(M, u)$, we have
\[
\log \frac{f(u)}{f(M)} \le - (1+\delta)  \log\frac{F(u)}{F(M)}
\]
for all $u\ge M$.
Then 
\begin{equation}
\label{S3L3E1}
\frac{1}{f(M)F(M)^{1+\delta}} < \frac{1}{f(u)F(u)^{1+\delta}}  = - \frac{F'(u)}{F(u)^{1+\delta}}
\end{equation}
for all $u\ge M$.
Since $ u^*(r) \to \infty$ as $r\to 0$, there exists $r_2>0$ such that
\[
u^*(r) > M \quad \mbox{for} \quad r\in (0,r_2).
\]
Integrating the above inequality with respect to $u$ over $(M, u^*(r))$ again, we have
\[
u^*(r) < \frac{f(M)F(M)^{1+\delta}}{\delta} (F(u^*(r))^{-\delta}- F(M)^{-\delta}) +M
\]
for $r\in (0,r_2)$.
Since $F(u^*(r)) = (2N-4)^{-1}r^2(1+o(1))$ as $r\to0$ (by \eqref{S2P1E3}),
there exist $R_0\in (0, r_2)$ and $C>0$ such that
\[
u^*(r) \le C r^{-2\delta}  \quad \mbox{for} \quad r\in (0,R_0).
\]

By \eqref{S3L3E1} we have
\begin{equation}\label{S3L3E2}
f(u^*(r)) < \frac{f(M)F(M)^{1+\delta}}{F(u^*(r))^{1+\delta}} \le C r^{-2-2\delta}
\end{equation}
for $r\in (0,R_0)$. Therefore, by (\ref{S2P1E2}) we see that
\[
-r^{N-1} (u^*)'(r) 
= \int_0^r f(u^*(s)) s^{N-1} \, \di s
\le C \int_0^r s^{N-3-2\delta} \, \di s \le C r^{N-2-2\delta}.
\]
Then we have 
\[
|(u^*)'(r)| \le C r^{-1-2\delta} \quad \mbox{for} \quad r\in (0,R_0).
\]
Thus, the proof is complete.
\end{proof}

\begin{proof}[Proof of Proposition~{\rm \ref{S1P1}}]
Let $u^*$ denote the unique positive singular solution of equation~\eqref{ODE} on $0<r<r_1$ given in Proposition~\ref{S2P1}.
We extend the domain of $u^*$ such that $u^*$ satisfies equation \eqref{ODE} for all $r>0$.

We show that $u^*(r)>0$ for $r>0$.
Let $P(u^*(r))$ be defined by
\[
P(u^*(r)):=\frac{1}{2}r^N(u^*)'(r)^2+r^NF_0(u^*(r))+\frac{N-2}{2}r^{N-1}u^*(r)(u^*)'(r)
\]
for $r>0$,
where
\[
F_0(u):=\int_0^uf(s) \, \di s.
\]
Using (\ref{ODE}), we have
\[
\frac{d}{dr}P(u^*(r))=-\frac{N-2}{2}r^{N-1}Q(u^*(r)),
\]
where $Q$ is defined in (A4).
Integrating the above equality over $(\rho,r)$, we have
\begin{equation}\label{S1P1E5}
P(u^*(r))-P(u^*(\rho))=-\frac{N-2}{2}\int_{\rho}^rs^{N-1}Q(u^*(s)) \, \di s.
\end{equation}
Let $\delta>0$ be small such that $N-2-2\delta>0$.
By L'Hospital's rule we have
\[
\lim_{\rho\to 0}\left|\frac{F_0(u^*(\rho))}{\rho^{-N}}\right|
=\lim_{\rho\to 0}\left|\frac{f(u^*(\rho))(u^*)'(\rho)}{-N\rho^{-N-1}}\right|
\le\lim_{\rho\to 0}C\rho^{N-2-2\delta}=0,
\]
where Lemma~\ref{S3L3} and (\ref{S3L3E2}) are used.
We have
\begin{align*}
|P(u^*(\rho))|
&\le C\rho^{N-2-2\delta}+|\rho^NF_0(u^*(\rho))|+C\rho^{N-2-2\delta}\\
&\to 0\quad \mbox{as}\quad \rho\to 0.
\end{align*}
Assume that there exists $r_0>0$ such that $u^*(r_0)=0$.
Then, by Hopf's lemma, $(u^*)'(r_0)<0$.
Letting $\rho\to 0$ in (\ref{S1P1E5}) with $r=r_0$, we have
\[
0<P(u^*(r_0))-0=-\frac{N-2}{2}\int_0^{r_0}s^{N-1}Q(u^*(s)) \, \di s\le 0,
\]
which is a contradiction.
Thus, $r_0=\infty$, and hence $u^*(r)>0$ for $0\le r<\infty$.
It is obvious that the other properties (\ref{SS}) are satisfied.

Hereafter, we show that $u^*\in\calL^1_{\ul}(\R^N)$.
Let $u_n(x):=\min\{n,u^*(x)\}$, $n\ge 1$.
Then it is obvious that $u_n\in BUC(\RN)$.
Let $\delta>0$ be small such that $N-2\delta>0$.
Since $N-1-2\delta>-1$, by Lemma~\ref{S3L3} we have
\[
\int_{B(0,2)}|u^*(x)-u_n(x)|\,\di x\le C\int_0^2r^{N-1-2\delta}\,\di r<\infty.
\]
By the dominated convergence theorem we have
\[
\lim_{n\to\infty}\int_{B(0,2)}|u^*(x)-u_n(x)|\,\di x=0,
\]
which indicates that $\left\|u^*-u_n\right\|_{L^1_{\ul}(\RN)}\to 0$ as $n\to\infty$.
Thus, $u^*\in\calL^1_{\ul}(\RN)$.
The proof is complete.
\end{proof}

\begin{lemma}\label{S3L4}
Suppose that $N\ge 3 $ and that {\rm (A1)}--{\rm (A3)} hold.
Let $u^*$ be the positive  singular solution of equation \eqref{ODE} on $0<r<r_1$ given in Proposition~{\rm \ref{S2P1}}.
For any small $\delta>0$, there exist $R_0\in (0,r_1)$ and $C>0$ such that
\[
 f(u^*(r)) > Cr^{-2+2\delta}
\]
for $r\in (0,R_0)$.
\end{lemma}
\begin{proof} 
Since 
\[
\lim_{u\to\infty} f'(u) F(u) =1, % < q_S,
\]
there exists $M>0$ such that $f'(u) F(u) \ge 1-\delta$ for all $u\ge M$.
This implies that
\[
\frac{f'(u)}{f(u)} \ge \frac{1-\delta}{f(u) F(u)} = -\frac{(1-\delta) F'(u)}{F(u)}
\]
for all $u\ge M$.
Integrating the above inequality with respect to $u$ over $(M, u)$, we have
\[
\log \frac{f(u)}{f(M)} \ge - (1-\delta)  \log\frac{F(u)}{F(M)}
\]
for all $u\ge M$.
Then 
\begin{equation*}
\frac{1}{f(M)F(M)^{1+\delta}} >\frac{1}{f(u)F(u)^{1-\delta}}  = - \frac{F'(u)}{F(u)^{1-\delta}}
\end{equation*}
for all $u\ge M$.
Since $ u^*(r) \to \infty$ as $r\to 0$, there exists $R_0'>0$ such that
\[
u^*(r) > M \quad \mbox{for} \quad r\in (0,R_0').
\]
Since $F(u^*(r)) = (2N-4)^{-1}r^2(1+o(1))$ as $r\to0$,
there exist $R_0\in (0, R_0')$ and $C>0$ such that
\begin{equation*}
f(u^*(r)) > \frac{f(M)F(M)^{1-\delta}}{F(u^*(r))^{1-\delta}} \ge C r^{-2+2\delta}
\end{equation*}
for $r\in (0,R_0)$. 
Thus, the proof is complete.
\end{proof}

\begin{lemma}\label{S3L5}
Suppose that $N\ge 3 $ and that {\rm (A1)}--{\rm (A3)} hold.
Let $u^*$ be the singular solution of equation \eqref{ODE} given in Proposition~{\rm \ref{S2P1}}.
For $\delta>0$, there exist $\gamma_0\in (0,1)$ and $R_0>0$ such that
\[
f(u^*(x)-\delta)\le\gamma_0f(u^*(x))\ \ \textrm{for}\ |x|<R_0.
\]
\end{lemma}
\begin{proof}
We can choose a small $R_0>0$ such that $g(\tau)$ is convex for $\tau>u^*(R_0)-\delta$.
We have
\[
f(u^*(x)-\delta)=e^{g(u^*(x)-\delta)}\le e^{g(u^*(x))-g'(u^*(x)-\delta)\delta}
=f(u^*(x))e^{-g'(u^*(x)-\delta)\delta}%\ \ \textrm{for}\ |x|<R_0.
\]
for $|x|<R_0$.
Since $g'(\tau)\ge C>0$ for $\tau>u^*(R_0)-\delta$, we see that $e^{-g'(u^*(x)-\delta)\delta}\le e^{-C\delta}$.
Let $\gamma_0:=e^{-C\delta}$.
Then the conclusion of the lemma follows.
\end{proof}

\begin{lemma}\label{S3L6}
Suppose that $N\ge 3 $ and that {\rm (A1)}--{\rm (A3)} hold.
Let $u^*$ be the singular solution of equation \eqref{ODE} given in Proposition~{\rm \ref{S2P1}}.
For each $\gamma_1\in (0,1)$, there exists $C>0$ and $R_2>0$ such that
\[
f(\gamma_1u^*(x))\le\frac{C}{|x|^{2\gamma_1}}\ \ \textrm{for}\ |x|<R_2.
\]
\end{lemma}
\begin{proof}
We choose $R_2>0$ such that $g(\tau)$ is convex and $g'(\tau)>0$ for $\tau\ge \gamma_1u^*(R_2)$.
For $\tau\ge 0$, we define a function $\tilde{g}(\tau)$ such that $\tilde{g}(\tau)=g(\tau)-g(0)$ for $\tau>\gamma_1u^*(R_2)$ and  $\tilde{g}(\tau)$ is convex for $\tau\ge 0$.
Since $0<\gamma_1<1$, we have $\tilde{g}(\gamma_1u^*(x))\le\gamma_1\tilde{g}(u^*(x))$ for $|x|<R_2$.
Therefore,
\begin{equation*}
\begin{split}
f(\gamma_1u^*)=e^{g(0)}e^{\tilde{g}(\gamma_1 u^*)}
\le e^{g(0)}e^{\gamma_1\tilde{g}(u^*)}
&\le e^{g(0)}\frac{g'(u^*)^{\gamma_1}}{g'(u^*(R_2))^{\gamma_1}}e^{\gamma_1g(u^*)}\\
&=C\left(f'(u^*)\right)^{\gamma_1}
\le \frac{C}{F(u^*)^{\gamma_1}}\le \frac{C}{|x|^{2\gamma_1}}
\end{split}
\end{equation*}
for $|x|<R_2$,
where the last inequality follows from \eqref{S2P1E3}.
\end{proof}

%%%%%%%%%%%%%%%%%%%%%%%%%%%%%%%%%%%%%%%%%%%%%%%%%%%%%%%
%%%%%%%%%%%%%%%%%%%%%%%%%%%%%%%%%%%%%%%%%%%%%%%%%%%%%%%
%%%%%%%%%%%%%%%%%%%%%%%%%%%%%%%%%%%%%%%%%%%%%%%%%%%%%%%
% Section 4
%%%%%%%%%%%%%%%%%%%%%%%%%%%%%%%%%%%%%%%%%%%%%%%%%%%%%%%
%%%%%%%%%%%%%%%%%%%%%%%%%%%%%%%%%%%%%%%%%%%%%%%%%%%%%%%
%%%%%%%%%%%%%%%%%%%%%%%%%%%%%%%%%%%%%%%%%%%%%%%%%%%%%%%
\section{\bf Singular stationary solution of problem~\eqref{S1E1}}
Later, we use the following:
\begin{proposition}[{\cite[Proposition 2.2]{MT06}}]\label{S4P1}
For $1\le p<\infty$, let $\calL^p_{\ul}(\RN)$ be defined by \eqref{S3E0}.
The following assertions are equivalent:\vspace{3pt}\\
{\rm (i)} $u\in\calL^p_{\ul}(\RN)$;\vspace{3pt}\\
{\rm (ii)} $\displaystyle \lim_{|y|\to 0}\left\|u(\,\cdot\,+y)-u(\,\cdot\,)\right\|_{L^p_{\ul}(\RN)}=0$;\\
{\rm (iii)} $\displaystyle \lim_{t\to 0}\left\| S(t)u-u\right\|_{L^p_{\ul}(\RN)}=0$.
\end{proposition}

\begin{lemma}\label{S4L1}
Let $u^*$ be as in Proposition~{\rm \ref{S1P1}}.
Then $u^*$ is a singular stationary solution of problem~\eqref{S1E1} in the sense of Definition~{\rm \ref{S1D1}}.
In particular, 
\[
\infty>u^*(x)=[S(t)u^*](x)+\int_0^t[S(t-s)f(u^*)](x)\, \di s
\]
for all $x\in\RN\setminus\{0\}$ and $t\in(0,\infty)$.
\end{lemma}

\begin{proof}
Let $t\in(0,\infty)$.
We show that $u(x,t)=u^*(x)$ is a solution in the sense of Definition~\ref{S1D1}.
Since $u^*(x)\in C^2(\RN\setminus\{0\})$, we see by Green's identity that, for $s\in (0,t)$,
\begin{equation}
\label{S4L1E1}
\begin{split}
&\int_{B(0,\e)^c}G(x-y,t-s)\Delta u^*(y) \, \di y\\
& =\int_{\partial B(0,\e)^c}\left(
G(x-y,t-s)\frac{\partial}{\partial\nu_y}u^*(y)-u^*(y)\frac{\partial}{\partial\nu_y}G(x-y,t-s)\right) \,\di \sigma(y)\\
&\qquad+\int_{B(0,\e)^c}\Delta_yG(x-y,t-s)u^*(y) \,\di y\\
&=:I_1+I_2.
\end{split}
\end{equation}
By Lemma~\ref{S3L3} we have
\begin{align}\label{S4L1E2}
|I_1|
&\le C\left|\left.\frac{\partial}{\partial\nu}u^*(y)\right|_{|y|=\e}\right|\e^{N-1}
+C\left|\left.u^*(y)\right|_{|y|=\e}\right|\e^{N-1}\nonumber\\
&\le C\e^{\frac{N}{2}-1-2\delta}+C\e^{\frac{N}{2}-2\delta}
\to 0\ \ \textrm{as}\ \ \e\to 0,
\end{align}
since ${N/2}-1-2\delta>0$ and ${N/2}-2\delta>0$ for small $\delta>0$.
Integrating (\ref{S4L1E1}) with respect to $s$ over $(0,t-\delta)$ , we have
\begin{equation*}
\begin{split}
&\int_0^{t-\delta}\int_{B(0,\e)^c}G(x-y,t-s)\Delta u^*(y) \, \di y\di s\\
&\qquad=\int_0^{t-\delta}I_1\,\di s-\int_0^{t-\delta}\int_{B(0,\e)^c}\partial_sG(x-y,t-s)u^*(y)\,\di y\di s\\
&\qquad=\int_0^{t-\delta}I_1\,\di s-\int_{B(0,\e)^c}\int_0^{t-\delta}\partial_sG(x-y,t-s)u^*(y)\,\di s\di y\\
&\qquad=\int_0^{t-\delta}I_1\,\di s-\int_{B(0,\e)^c}\left(G(x-y,\delta)u^*(y)-G(x-y,t)u^*(y)\right)\,\di y,
\end{split}
\end{equation*}
where we used $\Delta_yG(x-y,t-s)=-\partial_sG(x-y,t-s)$.
Letting $\e\to 0$, by the dominated convergence theorem and (\ref{S4L1E2}) we have
\begin{equation}\label{S4L1E3}
\lim_{\e\to 0}\int_0^{t-\delta}\int_{B(0,\e)^c}G(x-y,t-s)\Delta u^*(y) \, \di y\di s
=-S(\delta)u^*+S(t)u^*.
\end{equation}
By the monotone convergence theorem we have
\begin{equation}
\label{S4L1E4}
\begin{split}
&\lim_{\e\to 0}\int_0^{t-\delta}\int_{B(0,\e)^c}G(x-y,t-s)f(u^*(y))\,\di y\di s\\
&\qquad=\lim_{\e\to 0}\int_0^{t-\delta}\int_{\RN}G(x-y,t-s)f(u^*(y))\chi_{B(0,\e)^c}(y)\,\di y\di s\\
&\qquad=\int_0^{t-\delta}\int_{\RN}G(x-y,t-s)f(u^*(y))\,\di y\di s.
\end{split}
\end{equation}
The solution $u^*$ satisfies $\Delta u^*+f(u^*)=0$ in the classical sense for $x\in B(0,\e)^c$.
Hence,
\begin{equation*}
\begin{split}
&\int_0^{t-\delta}\int_{B(0,\e)^c}G(x-y,t-s)\Delta u(y)\,\di y\di s\\
&\qquad\qquad+\int_0^{t-\delta}\int_{B(0,\e)^c}G(x-y,t-s)f(u^*(y))\,\di y\di s=0.
\end{split}
\end{equation*}
Letting $\e\to 0$, by (\ref{S4L1E3}) and (\ref{S4L1E4}) we have
\begin{equation}\label{S4L1E5}
-S(\delta)u^*+S(t)u^*+\int_0^{t-\delta}S(t-s)f(u^*) \, \di s=0.
\end{equation}
Since $u^*\in\calL^1_{\ul}(\RN)$  (by Lemma~\ref{S3L4}), by Proposition~\ref{S4P1} with $p=1$ we have
\begin{equation}\label{S4L1E6}
\lim_{\delta\to 0}\left\|S(\delta)u^*-u^*\right\|_{L^1_{\ul}(\RN)}=0.
\end{equation}
By the monotone convergence theorem we have
\begin{equation}
\label{S4L1E7}
\begin{split}
\lim_{\delta\to 0}\int_0^{t-\delta}S(t-s)f(u^*)\,\di s
&=\lim_{\delta\to 0}\int_0^t\left(S(t-s)f(u^*)\right)\chi_{(0,t-\delta)}(s)\, \di s\\
&=\int_0^tS(t-s)f(u^*)\,\di s.
\end{split}
\end{equation}
By \eqref{S4L1E7}, \eqref{S4L1E6}, and \eqref{S4L1E5} we see that as $\delta\to 0$,
\[
u^*=S(t)u^*+\int_0^tS(t-s)f(u^*)\,\di s
\]
for all $(x,t)\in(\R^N\setminus\{0\})\times (0,\infty)$.
Thus, $u^*$ is a solution in the sense of Definition~\ref{S1D1}.
The proof is complete.
\end{proof}
Theorem~\ref{S1T1}~(ii) immediately follows from Lemma~\ref{S4L1}.

%%%%%%%%%%%%%%%%%%%%%%%%%%%%%%%%%%%%%%%%%%%%%%%%%%%%%%%
%%%%%%%%%%%%%%%%%%%%%%%%%%%%%%%%%%%%%%%%%%%%%%%%%%%%%%%
%%%%%%%%%%%%%%%%%%%%%%%%%%%%%%%%%%%%%%%%%%%%%%%%%%%%%%%
% Section 5
%%%%%%%%%%%%%%%%%%%%%%%%%%%%%%%%%%%%%%%%%%%%%%%%%%%%%%%
%%%%%%%%%%%%%%%%%%%%%%%%%%%%%%%%%%%%%%%%%%%%%%%%%%%%%%%
%%%%%%%%%%%%%%%%%%%%%%%%%%%%%%%%%%%%%%%%%%%%%%%%%%%%%%%
%\newpage
\section{\bf Existence of a solution of problem~\eqref{S1E1}}
\begin{definition}\label{S5D1}
{\rm (i)} We call $v(t)$ a minimal solution of problem~\eqref{S1E1} if $v(t)$ is a solution of problem~\eqref{S1E1} in the sense of Definition~{\rm\ref{S1D1}} such that for any solution $u(t)$, $0\le u\le u^*$ a.e.~in $\RN\times [0,T)$, in the sense of Definition~{\rm\ref{S1D1}}, $v\le u$ a.e.~in $\RN\times [0,T)$.\\
{\rm (ii)} Let $w(t)\le u^*$.
We call $w(t)$ a maximal solution of problem~\eqref{S1E1} if $w(t)$ is a solution of problem~\eqref{S1E1} in the sense of Definition~{\rm\ref{S1D1}} such that for any solution $u(t)$, $0\le u\le u^*$ a.e.~in $\RN\times [0,T)$, in the sense of Definition~{\rm\ref{S1D1}}, $w\ge u$ a.e.~in $\RN\times [0,T)$.
\end{definition}

We construct two functions satisfying \eqref{S1D1E2}, which will become the minimal and maximal solutions.
\begin{lemma}\label{S5L0}
Suppose that $0\le u_0\le u^*$ a.e.~in $\mathbb{R}^N$.
Then there exists nonnegative measurable functions $v(t)$ and $w(t)$ on $\mathbb{R}^N\times[0,\infty)$ such that \eqref{S1D1E2} holds for a.a.~$(x,t)\in\R^N\times[0,\infty)$ and that $v\le w\le u^*$ a.e.~in $\RN\times[0,\infty)$.
\end{lemma}
\begin{proof}
We define
\begin{equation}\label{wk}
w_0=u^*,\quad
w_k(t)=S(t)u_0+\int_0^tS(t-s)f(w_{k-1}(s))\,\di s,
\end{equation}
and
\begin{equation}\label{vk}
v_0=0,\quad
v_k(t)=S(t)u_0+\int_0^tS(t-s)f(v_{k-1}(s))\,\di s,
\end{equation}
for $k=1,2,\cdots$.
We see from  Lemma~\ref{S4L1} that $v_0\le v_1\le w_1\le w_0$.
By induction we can easily deduce
\begin{equation}
\label{S5L0E1}
0=v_0\le v_1\le v_2\le\cdots\le w_2\le w_1\le w_0=u^*<\infty
\end{equation}
for a.a.~$(x,t)\in\R^N\times [0,\infty)$.
For each $t>0$, the limits $\lim_{k\to\infty} v_k$ and $\lim_{k\to\infty} w_k$ exist for a.a.~$(x,t)\in\RN\times [0,\infty)$.
Thus, we define
\[
v(t):=\lim_{k\to\infty} v_k(t),\quad w(t):=\lim_{k\to\infty} w_k(t),
\]
for a.a.~$(x,t)\in\RN\times[0,\infty)$.
We apply the monotone convergence theorem to \eqref{wk}.
Then we obtain
\[
\infty>w(t)=S(t)u_0+\int_0^tS(t-s)f(w(s))\,\di s
\]
for a.a.~$(x,t)\in\RN\times [0,\infty)$.
and hence $w$ satisfies \eqref{S1D1E2} for a.a.~$(x,t)\in \R\times [0,\infty)$.
Thus, the proof is complete.
\end{proof}

\begin{lemma}\label{S5L0+}
Let $v(t)$ and $w(t)$ be functions constructed in Lemma~{\rm\ref{S5L0}}.
Then, $v(t)$ and $w(t)$ are the minimal and maximal solutions in the sense of Definition~{\rm\ref{S5D1}}, respectively.
\end{lemma}
\begin{proof}
Let $u(t)$ be a nonegative solution of problem~\eqref{S1E1} in the sense of Definition~\ref{S1D1} such that $0\le u(t)\le u^*$ a.e.~in $\RN\times[0,T)$. Then,
\[
u(t)=S(t)+\int_0^tS(t-s)f(u(s)) \di s.
\]
Since $v_0=0\le u(t)$ a.e.~in $\R^N\times [0,T)$, by induction we see that for each $k\ge 1$,
\[
u(t)-v_k(t)=\int_0^tS(t-s)\left\{f(u(s))-f(v_{k-1}(s))\right\}\di s\ge 0
\]
a.e.~in $\R^N\times [0,T)$.
Since $v(t)=\lim_{k\to\infty}v_k(t)\le u(t)$, $u(t)$ is a minimal solution of \eqref{S1E1}.
The proof of the maximal solution is similar.
We omit the details.
\end{proof}

We use the following proposition in the proof of Lemma~\ref{S5L1} below:
\begin{proposition}[{\cite[Corollary 3.1]{MT06}}]\label{S5P1}
Let $1\le p\le q\le\infty$.
Then there exists $C_1=C_1(N,p,q)>0$ such that
\[
\left\|S(t)u\right\|_{L^q_{\ul}(\R^N)}\le C_1\left(t^{-\frac{N}{2}\left(\frac{1}{p}-\frac{1}{q}\right)}+1\right)\left\|u\right\|_{L^p_{\ul}(\R^N)}
\]
for $t>0$ and $u\in L^p_{\ul}(\R^N)$.
\end{proposition}

\begin{lemma}\label{S5L1}
Suppose that $0\le u_0\le u^*$ a.e.~in $\R^N$ and $u_0\not\equiv u^*$ a.e.~in $\mathbb{R}^N$ (and hence $u_0<u^*$ on a set of positive measure in $\R^N$).
Then the maximal solution $w$ of problem~\eqref{S1E1} constructed in Lemma~{\rm\ref{S5L0}} is a solution in $\mathbb{R}^N\times[0,\infty)$ in the sense of Definition~{\rm \ref{S1D1}} with $T=\infty$ satisfying $w\in L_{loc}^{\infty}(0,\infty;L^{\infty}(\RN))$.
Therefore, $w$ satisfies the equation in \eqref{S1E1} in $\RN\times (0,\infty)$ in the classical sense.
\end{lemma}

\begin{proof}
In Lemma~\ref{S5L0+} we have already shown that $w$ is the maximal solution.
All we have to do is to prove that $w\in L^{\infty}_{loc}(0,\infty;L^{\infty}(\RN))$.
We choose $T>0$ arbitrarily.

We divide the proof into five steps.\\
{\it Step 1:} We show that
\begin{equation}\label{S5L1E0}
w_1(x,t)=u^*(x)-\eta (x,t),
\end{equation}
where $\eta(x,t)$ satisfies
\begin{equation}\label{S5L1E1}
\eta(x,t)=[S(t)(u^*-u_0)](x)\ \ \textrm{a.e.~in}\ \RN\ \textrm{for all}\ 0\le t<T.
\end{equation}
By Lemma~\ref{S4L1},
\begin{equation}\label{S5L1E2}
u^*=S(t)u^*+\int_0^tS(t-s)f(u^*)\, \di s.
\end{equation}
Subtracting (\ref{S5L1E1}) from (\ref{S5L1E2}), we have
\[
u^*-\eta(t)=S(t)\left[u^*-(u^*-u_0)\right]+\int_0^tS(t-s)f(u^*)\, \di s=w_1,
\]
where the last equality follows from (\ref{wk}) with $k=1$.

\noindent
{\it Step 2:} We show that, for each $t_0\in(0,T/2)$, there exist $R_0>0$ and $\delta>0$ such that
\[
w_2(x,t)\le u^*(x)-\delta\ \textrm{a.e.~in $B(0,R_0)$ for all $t_0\le t<T$.}
\]

It follows from (\ref{S5L1E1}) that $\eta\ge 0$ a.e.~in $\R^N$ for $t\ge 0$.
Because of (\ref{S5L1E0}), we see that $w_1(t)\le u^*$ a.e.~in $\R^N$ for $0\le t<T$.
Then, it is obvious that $w_1$ is a supersolution of \eqref{wk} with $k=2$.
Hence, $0\le w_2(t)\le w_1(t)$ a.e.~in $\RN$ for $0\le t<T$.
Because of the assumption of the lemma, $u_0\le u^*$ and $u_0\not\equiv u^*$ a.e.~in $\R^N$, and hence $u_0< u^*$ on a set of positive measure.
Because of (\ref{S5L1E1}), for each $t_0\in (0,T)$, there exist $R_0>0$ and $\delta>0$ such that
$\eta(x,t)\ge\delta$ a.e.~in $B(0,R_0)$ for all $t_0<t<T$.
Then,
\[
0\le w_2(x,t)\le w_1(x,t)=u^*(x)-\eta(x,t)\le u^*(x)-\delta
\ \ \textrm{a.e.~in $B(0,R_0)$ for all $t_0\le t<T$.}
\]

\noindent
{\it Step 3:} We show that there exists a small $\e>0$ such that
\begin{equation}\label{S5L1E4-}
    \sup_{t_0\le t<T}\left\| f(w_3(x,t))\right\|_{L^{\frac{N}{2}+\e}_{\ul}(\RN)}<\infty.
\end{equation}

By Lemma~\ref{S3L5} there exists $0<\gamma_0<1$ and $R_1\in (0,R_0]$ such that $f(u^*(x)-\delta)\le\gamma_0f(u^*(x))$ a.e.~in $B(0,R_1)$.
By Step 2 we see that
\[
f(w_2(x,t))\le f(u^*(x)-\delta)\le\gamma_0f(u^*(x))\ \ \textrm{a.e.~in $B(0,R_1)$ for all $t_0\le t<T$}.
\]
We define
\begin{equation*}
\begin{split}
z_0(t)&:=S(t)u_0+\int_0^tS(t-s)\left[f(u^*)\chi_{B(0,R_1)^c}(y)\right]\,\di s,\\
z_1(t)&:=\int_0^tS(t-s)\left[\gamma_0f(u^*)\chi_{B(0,R_1)}(y)\right]\,\di s.
\end{split}
\end{equation*}
It is obvious from Lemma~\ref{S4L1} that $z_1(t)\le\gamma_0 u^*$ a.e.~in $\R^N$.
Let $\bw:=z_0+z_1$. Since
\begin{equation*}
\begin{split}
&\bw-S(t)u_0-\int_0^tS(t-s)f(w_2(s))\,\di s\\
&=\int_0^tS(t-s)\left[f(u^*)\chi_{B^c_{R_1}}+\gamma_0f(u^*)\chi_{B_{R_1}}\right]\di s
-\int_0^tS(t-s)f(w_2(s))\, \di s\ge 0,
\end{split}
\end{equation*}
$\bw$ is a supersolution of \eqref{wk} with $k=3$, and hence $w_3\le\bw$ a.e.~in $\RN$ for $0\le t<T$.
Since
\[
\sup_{t_0\le t<T}\left\|S(t)u_0\right\|_{L^{\infty}(\RN)}<\infty\quad\mbox{and}\quad
\sup_{0\le t<T}\left\|f(u^*)\chi_{B(0,R_1)^c}\right\|_{L^{\infty}(\RN)}<\infty,
\]
by the definition of $z_0(t)$ we see that $\sup_{t_0\le t<T}\left\|z_0(t)\right\|_{L^{\infty}(\RN)}<\infty$.
Since
\[
w_3(x,t)\le \bw(t)=z_0(t)+z_1(t)\le\gamma_0u^*(x)+C
\]
a.e.~in $B(0,R_1)$ for all $t_0<t<T$,
there exist $\gamma_1\in(\gamma_0,1)$ and $R_2\in (0,R_1]$ such that $w_3(x,t)\le\gamma_1u^*(x)$ a.e.~in $B(0,R_2)$ for all $t_0\le t<T$.
By  Lemma~\ref{S3L6} that there exists $R_3\in (0,R_2]$ such that
\[
f(w_3(x,t))\le \frac{C}{|x|^{2\gamma_1}}\ \ \textrm{a.e.~in $B(0,R_3)$ for all $t_0\le t<T$}.
\]
Since there exists a small $\e>0$ such that $-\gamma_1 N-2\gamma_1\e+N-1>-1$, we have
\begin{equation}
\label{S5L1E4}
\begin{split}
\int_{B(0,R_3)}\left|f(w_3(x,t))\right|^{\frac{N}{2}+\e} \, \di x
&\le \int_{B(0,R_3)}\left(\frac{C}{|x|^{2\gamma_1}}\right)^{\frac{N}{2}+\e} \, \di x\\
&\le C\int_0^{R_3}r^{-2\gamma_1\left(\frac{N}{2}+\e\right)}r^{N-1}\,\di r<\infty.
\end{split}
\end{equation}
Here, $R_3$ may be less than $1$ which is the radius of $B(0,1)$ in the definition of the norm of $L^p_{\ul}(\R^N)$.
However, $f(w_3(x,t))$ is essentially bounded on $B(0,R_3)^c\times [t_0,T)$.
Thus,
by \eqref{S5L1E4} we see that
$\sup_{t_0\le t<T}\left\| f(w_3(x,t))\right\|_{L^{{N/2}+\e}_{\ul}(\RN)}<\infty$.

\noindent
{\it Step 4:} We show that
\begin{equation}\label{S5L1E4+}
\sup_{2t_0\le t<T}\left\|w_4(x,t)\right\|_{L^{\infty}(\RN)}<\infty.
\end{equation}

Because of the definition of $w_4$,
\begin{equation}\label{S5L1E5}
w_4(t_0)=S(t_0)u_0+\int_0^{t_0}S(t_0-s)f(w_3(s)) \, \di s.
\end{equation}
Using (\ref{S5L1E5}), we have
\begin{equation*}
\begin{split}
&w_4(t)  -S(t-t_0)w_4(t_0)-\int_{t_0}^tS(t-s)f(w_3(s)) \, \di s\\
&=w_4(t)-S(t-t_0)\left[
S(t_0)u_0+\int_0^{t_0}S(t_0-s)f(w_3(s)) \, \di s\right]
-\int_{t_0}^tS(t-s)f(w_3(s)) \, \di s\\
&=w_4(t)-S(t)u_0-\int_0^{t_0}S(t-s)f(w_3(s)) \, \di s-\int_{t_0}^tS(t-s)f(w_3(s)) \, \di s\\
&=w_4(t)-S(t)u_0-\int_0^tS(t-s)f(w_3(s)) \,\di s=0
\end{split}
\end{equation*}
for $t\in(t_0,T)$,
where the last equality follows from the definition of $w_4$.
Thus, by Proposition~\ref{S5P1} we have
\begin{equation*}
\begin{split}
&\left\|w_4(t)\right\|_{L^{\infty}(\RN)}\\
&\le\left\|S(t-t_0)w_4(t_0)\right\|_{L^{\infty}(\RN)}
+\int_{t_0}^t\left\|S(t-s)f(w_3(s))\right\|_{L^{\infty}(\RN)} \, \di s\\
&\le C(t-t_0)^{-N/2}\left\|w_4(t_0)\right\|_{L^1_{\ul}(\RN)}+
\int_{t_0}^tC\left\{(t-s)^{-\frac{N}{2}\frac{1}{\frac{N}{2}+\e}}+1\right\} \, \di s
\sup_{t_0\le t<T}\left\|f(w_3(t))\right\|_{L^{\frac{N}{2}+\e}_{\ul}(\RN)}\\
&\le C\quad\mbox{uniformly for $t\in[2t_0,T)$.}
\end{split}
\end{equation*}
Here, for each fixed $T$, $C$ does not depend on $t$, but $C$ depends on $T$.
This indicates that $\sup_{2t_0\le t<T}\left\|w_4(t)\right\|_{L^{\infty}(\RN)}<\infty$.

\noindent
{\it Step 5:} Conclusion.

Let $w$ be the maximal solution of problem~\eqref{S1E1} which is constructed in Lemma~\ref{S5L0+}.
In Lemma~\ref{S5L1} we already showed that $w$ is a solution in $\RN\times [0,T)$ in the sense of Definition~\ref{S1D1} with $T>0$.
Then, by \eqref{S5L0E1} we have
\[
w(t)\le w_4(t)\ \text{a.e.~in $\RN\times [0,T)$.}
\]
The conclusion of the lemma holds, since \eqref{S5L1E4+} holds for arbitrary small $t_0>0$.
Since $T>0$ is arbitrary,
$w\in L_{loc}^{\infty}(0,\infty;L^{\infty}(\R^N))$, and hence it follows from a parabolic regularity theorem that $w$ is a classical solution for $t\in(0,\infty)$.
Thus, the proof of Lemma~\ref{S5L1} is complete.
\end{proof}
Theorem~\ref{S1T1}~(i) follows from Lemma~\ref{S5L1}.

%%%%%%%%%%%%%%%%%%%%%%%%%%%%%%%%%%%%%%%%%%%%%%%%%%%%%%%
%%%%%%%%%%%%%%%%%%%%%%%%%%%%%%%%%%%%%%%%%%%%%%%%%%%%%%%
%%%%%%%%%%%%%%%%%%%%%%%%%%%%%%%%%%%%%%%%%%%%%%%%%%%%%%%
% Section 6
%%%%%%%%%%%%%%%%%%%%%%%%%%%%%%%%%%%%%%%%%%%%%%%%%%%%%%%
%%%%%%%%%%%%%%%%%%%%%%%%%%%%%%%%%%%%%%%%%%%%%%%%%%%%%%%
%%%%%%%%%%%%%%%%%%%%%%%%%%%%%%%%%%%%%%%%%%%%%%%%%%%%%%%
%\newpage
\section{\bf Nonexistence of solutions of problem~\eqref{S1E1}}

\subsection{Definition of weak solutions}
%%%%%%%%%%%%%%%%%%%%%%%%%%%%%%%%%%%%%
The proof of Theorem~\ref{S1T1}~(iii) is provided in the framework of weak solutions.
First, we introduce the definition of weak solutions of problem~\eqref{S1E1}.

\begin{definition}[{Weak solutions on $\R^N$}]\label{S6D1}
Let $u$ be a nonnegative measurable function on $\mathbb{R}^N \times(0,T)$, where $T\in (0,\infty]$
and $u_0 \in L^1_{loc} (\mathbb{R}^N)$.
We say that $u$ with $ f(u) \in L^1_{loc} (\mathbb{R}^N\times [0,T))$ is a weak solution of problem~\eqref{S1E1} if
$u$ satisfies
\begin{equation}\label{S6D1E1}
\begin{split}
&\int_{0}^{{\tau}} \int_{\mathbb{R}^N} u(-\partial_t -\Delta) \varphi \, \di x\di t\\
&\qquad\qquad = \int_{0}^{{\tau}}\int_{\mathbb{R}^N} f(u) \varphi \, \di x \di t + \int_{\mathbb{R}^N} u_0 \varphi(0)\, \di x
\end{split}
\end{equation}
for  any $0<\tau<T$ and any $0\le \varphi \in C^{2,1}_0 (\mathbb{R}^N\times[0,\tau])$
such that $\varphi(\tau) = 0$.
If  $u$ satisfies \eqref{S6D1E1} with $=$ replaced by $\ge$, then we say that $u$ is a weak supersolution in $\mathbb{R}^N\times [0,T)$.
\end{definition}
We note that any solution of problem~\eqref{S1E1} in the sense of Definition~\ref{S1D1} is also a weak solution in the sense of Definition~\ref{S6D1}.
In particular, the singular stationary solution of problem~\eqref{S1E1} is a weak solution in the sense of Definition \ref{S6D1}.
For details, see Theorem~\ref{S7T1} in Appendix below. 

%
%%%%%%%%%%%%%%%%%%%%%%%%%%%%%%%%%%%%%
\subsection{Problem~\eqref{S1E1} in balls.}
%%%%%%%%%%%%%%%%%%%%%%%%%%%%%%%%%%%%%
Let $R>0$. Consider the Cauchy-Dirichlet problem
\begin{equation}\label{S6E1}
	\left\{
	\begin{aligned}
		&\partial_t w  -\Delta  w =  f_{R}(w),	\quad && x\in B(0,R),\,\, t>0,\\
		&w=0,	\quad && x\in \partial B(0,R),\,\, t>0,\\
		&w(x,0) =  w_0(x),		\quad && x\in B(0,R),
	\end{aligned}
	\right.
\end{equation}
where $w_0$ is a nonnegative measurable function on $B(0,R)$ and $ f_{R}(w):= f(w+u^*(R))$.
Here, $u^*$ is the singular solution of problem~\eqref{SS} given in Proposition~\ref{S1P1}.
Obviously, $u^*_{R} := u^* - u^*(R)$ is a  singular stationary solution  of problem~\eqref{S6E1}.

For  $x,y\in \overline{B(0,R)}$ and $t>0$, let $G_R=G_R(x,y,t)$ be the Dirichlet heat kernel on $B(0,R)$.
We formulate the definition of solutions of problem~\eqref{S6E1}.

\begin{definition}\label{S6D2}
Let $w$ be a nonnegative measurable function on $B(0,R)\times(0,T)$, where $T\in(0,\infty]$.
\begin{itemize}
\item[(i)] We say that $w$ is a solution of problem~\eqref{S6E1} in $B(0,R)\times [0,T)$ if
\begin{equation}\label{S6D2E2}
\begin{split}
\infty>w(x,t)
&=\int_{B(0,R)}G_R(x,y, t )w_0(y) \, \di y\\
&\qquad+\int_0^t\int_{B(0,R)}G_R(x,y,t-s)f_R(w(y,s)) \, \di y\di s
\end{split}
\end{equation}
for a.a.~$(x,t)\in B(0,R)\times (0,T)$.
\item[(ii)] If  $u$ satisfies \eqref{S6D2E2} with $=$ replaced by $\ge$, then we say that $w$ is a supersolution in $B(0,R)\times [0,T)$.
\end{itemize}
\end{definition}

\begin{definition}
The minimal solution and the maximal solution of problem~\eqref{S6E1} are defined as in the case of whole space $\mathbb{R}^N$.
\end{definition}

\begin{proposition}\label{S6P4}
If a nonnegative measurable function $w_0$ on $B(0,R)$ satisfies
$w_0 \le u^*_R$ and $w_0 \not\equiv u^*_R$ a.e.~in $B(0,R)$, then problem~\eqref{S6E1} has the maximal solution $w$ in $B(0,R)\times[0,\infty)$ satisfying $w\in L^\infty_{loc}(0,\infty;L^\infty(B(0,R)))$.
\end{proposition}

\begin{proof}
By the assumption of Proposition~\ref{S6P4}, we can find the maximal solution of problem~\eqref{S6E1} in $B(0,R)\times[0,\infty)$. 
It suffices to prove  $w\in L^\infty_{loc}(0,\infty;L^\infty(B(0,R)))$.

Set $\overline{w}_0(x) := w_0 (x) \chi_{B(0,R)}(x)$.
Obviously, we see that
\[
\overline{w}_0(x) \le  u^*_R(x) \chi_{B(0,R)}(x) = (u^*(x) - u^*(R) )\chi_{B(0,R)}(x) \le u^*(x)
\]
for a.a.~$x\in \mathbb{R}^N$.
Note that $\overline{w}_0\not\equiv u^*$ a.e.~in $\mathbb{R}^N$.
It follows from Theorem~\ref{S1T1} (i) that
problem~\eqref{S1E1} with $u_0=\overline{w}_0$ has a solution $\overline{u}$ in $\mathbb{R}^N\times [0,\infty)$  in the sense of Definition~\ref{S1D1}
satisfying $\overline{u} \in L^\infty_{loc} (0,\infty; L^\infty(\mathbb{R}^N))$.
Since $\overline{u}$ is a supersolution of problem~\eqref{S6E1},
by the comparison principle,  we have $w(x,t) \le \overline{u} (x,t) $ for
a.a.~$(x,t)\in B(0,R)\times(0,\infty)$.
Furthermore, 
\[
\overline{u}  \in L^\infty_{loc} (0,\infty; L^\infty(\mathbb{R}^N)) \subset L^\infty_{loc} (0,\infty; L^\infty(B(0,R))). 
\]
Thus, we see that  $w\in L^\infty_{loc} (0,\infty; L^\infty(B(0,R)))$, and the proof is complete.
\end{proof}

%%%%%%%%%%%%%%%%%%%%%%%%%%%%%%%%%%%%%
\subsection{Proof of Theorem~\ref{S1T1} (iii).}
%%%%%%%%%%%%%%%%%%%%%%%%%%%%%%%%%%%%%

\begin{lemma}\label{S6L5}
Let $u$ be a solution of problem~\eqref{S1E1} in $\mathbb{R}^N\times[0,T)$ in the sense of Definition~{\rm \ref{S1D1}}, where $T\in (0,\infty]$.
Assume that there exists $R>0$ such that
\[
u(x,t) > \alpha u^*(x) \quad \mbox{for a.a.} \, (x,t)\in B(0,R)\times (t_1,T)
\]
for some $\alpha>1$  and $t_1\in [0,T)$. 
Then, for any $t_2\in (t_1, T)$, there exists $R'\in (0,R)$ such that
\[
u(x,t) > \left(\alpha+\frac{1}{2}\right) u^*(x)   \quad \mbox{for a.a.}\, (x,t)\in B(0,R') \times (t_2,T).
\]
\end{lemma}

\begin{proof}
We can assume  without loss of generality that $t_1=0$
and
$u_0(x)>\alpha u^*(x)$ for a.a.~$x\in B(0,R)$.
Set $\psi := u - \alpha u^* $.
We have $\psi \ge 0$ for a.a.~$(x,t)\in B(0,R)\times(0,T)$,
and by Theorem~\ref{S7T1} we see that $\psi$ satisfies
\begin{equation}
\label{eq:NE5}
\begin{split}
\infty&>\int_{0}^{{\tau}} \int_{\mathbb{R}^N} \psi(-\partial_t -\Delta) \varphi \, \di x\di t\\
& =\int_{\mathbb{R}^N} \psi_0 \varphi(0)\, \di x+ \int_{0}^{{\tau}} \int_{\mathbb{R}^N}[f(u)-\alpha f(u^*)] \varphi  \, \di x \di t \\
\end{split}
\end{equation}
for any $0<{\tau}<T$ and any $0\le \varphi \in  C^{2,1}_0 (\mathbb{R}^N\times[0,{\tau}])$ such that $\varphi({\tau})=0$,
where $\psi_0 = u_0 - \alpha u^*>0 $.
We show that 
\begin{equation}
\label{eq:6.6}
\left|\int_0^{{\tau}}\int_{\partial B(0,r)} \psi(y,t) \, \di\sigma(y) \, \di t\right| <\infty.
\end{equation}
for a.a.~$r>0$.
Indeed, if it is not the case, then there exist $0<R_0<R_1$ such that
\[
\left|\int_{R_0}^{R_1}\int_0^{{\tau}}\int_{\partial B(0,r)} \psi(y,t) \, \di\sigma(y) \di t \di r \right|=\left|\int_0^{{\tau}}\int_{B(0,R_1)\setminus B(0,R_0)}\psi(y,t)\,\di y\di t\right|=\infty,
\]
and hence $\psi\not\in L^1_{loc}(\mathbb{R}^N\times (0,T))$ which is a contradiction.
Thus, we can assume that \eqref{eq:6.6} with $r=R$ holds.

For $0\le h \in C^\infty_0 ({B(0,R)}\times[0,{\tau}))$, let $\varphi_{h,R}\in C^{2,1}_0 (\overline{B(0,R)}\times[0,{\tau}])$ 
be a solution of 
\begin{equation*}
	%\label{eq:AJP}
\left\{
	\begin{aligned}
		&-\partial_t \varphi_{h,R}  -\Delta  \varphi_{h,R} =  h,	\quad && x\in B(0,R),\,\, t\in (0,{\tau}),\\
		&\varphi_{h,R}(x,t) =  0,							\quad && x\in \partial B(0,R),\,\, t\in (0,{\tau}),\\
		&\varphi_{h,R}(x,{\tau}) =  0,									\quad && x\in B(0,R).
	\end{aligned}
	\right.
\end{equation*}
We have $\varphi_{h,R}\ge0$ for  $x\in B(0,R)$ and $t\in(0,{\tau})$
and $\left\|\varphi_{h,R}\right\|_{L^\infty(B(0,R)\times(0,{\tau}))} < \infty$.
We extend $\varphi_{h,R}$ and $h$ so that $\varphi_{h,R}(x,t) =  0$ and $h=0$ in $B(0,R)^c\times (0, {\tau})$.
Though   $\varphi_{h,R}\not\in C^{2,1}_0 (\mathbb{R}^N\times[0,{\tau}])$ (by Hopf's lemma), 
by using an approximation we  attempt to substitute $\varphi = \varphi_{h,R}$  into \eqref{eq:NE5} as a test function.
Let $\eta\in C^\infty_0(\mathbb{R}^N)$ be  such that $\eta$ is radially symmetric,
\[
\eta\ge0 \quad \mbox{in} \quad \mathbb{R}^N, \qquad \supp \eta \subset \overline{B(0,1)}, \qquad \int_{\mathbb{R}^N} \eta(x) \, \di x=1.
\]
For any $\e>0$, set 
\begin{equation*}
%\label{eq:MF}
\eta_\e (x) := \e^{-N} \eta \left(\frac{x}{\e}\right), \qquad x\in \mathbb{R}^N.
\end{equation*}
We have
\begin{equation}\label{eq:NE6-}
\varphi_{h,R}*\eta_\e \in C^{\infty}_0 ({B(0,R+\e)}\times[0,{\tau}])\subset C^{2,1}_0 (\mathbb{R}^N\times[0,{\tau}])
\end{equation}
with $\varphi_{h,R}*\eta_\e\ge0$ and $\varphi_{h,R}*\eta_\e ({\tau})=0$.
We can substitute $\varphi*\eta_\e$ into \eqref{eq:NE5} and have
\begin{equation}
\label{eq:NE6}
\begin{split}
&\int_{0}^{{\tau}} \int_{B(0,R+\e)} \psi(-\partial_t -\Delta) (\varphi_{h,R}*\eta_\e)\, \di x\di t\\
&=  \int_{B(0,R+\e)} \psi_0 (\varphi_{h,R}*\eta_\e)(0)\, \di x\\
&\qquad\qquad+ \int_{0}^{{\tau}}\int_{B(0,R+\e)} [f(u)- \alpha f( u^*)] (\varphi_{h,R}*\eta_\e) \, \di x \di t .\\
\end{split}
\end{equation}
Since $\psi_0 \in L^1_{loc}(\mathbb{R}^N)$ and $f(u), f(u^*) \in L^1_{loc}(\mathbb{R}^N\times [0,{\tau}])$,
it is easy to see that
\begin{equation*}
\begin{split}
 &\lim_{\e\to0}\int_{B(0,R+\e)} \psi_0 (\varphi_{h,R}*\eta_\e)(0)\, \di x
 =\int_{B(0,R)} \psi_0 \varphi_{h,R}(0)\, \di x,
 \\
 &\lim_{\e\to0}\int_{0}^{{\tau}}\int_{B(0,R+\e)} [f(u)- \alpha f( u^*)] (\varphi_{h,R}*\eta_\e) \, \di x \di t\\
 &\qquad\qquad\qquad\qquad  = \int_{0}^{{\tau}}\int_{B(0,R)} [f(u)- \alpha f( u^*)] \varphi_{h,R} \,\di x \di t. 
\end{split}
\end{equation*}
We consider the left hand side of \eqref{eq:NE6}.
For $x\in B(0,R-\e)$ and $t\in(0,T)$,
\begin{equation*}
\begin{split}
(-\partial_t -\Delta) (\varphi_{h,R}*\eta_\e) (x,t)
&= \int_{B(0,\e)} (-\partial_t -\Delta) \varphi_{h,R} (x-y,t) \cdot\eta_\e(y) \, \di y \\
&= \int_{B(0,\e)} h (x-y,t)\eta_\e(y) \, \di y \\
&\to h (x,t) \quad \mbox{as} \quad \e\to0.\\
\end{split}
\end{equation*}
Since $\psi \in L^1_{loc}(\mathbb{R}^N\times[0,{\tau}])$,
this implies that
\begin{equation*}
\begin{split}
&\lim_{\e\to0}\int_{0}^{{\tau}} \int_{B(0,R-\e)} \psi(-\partial_t -\Delta) (\varphi_{h,R}*\eta_\e)\, \di x\di t= 
\int_0^\tau\int_{B(0,R)} \psi(-\partial_t -\Delta) \varphi_{h,R}\, \di x\di t.
\end{split}
\end{equation*}
%
%For $x\in E_\e:= B(0,R+\e)\setminus B(0,R-\e)$, by Green's identity, we have
Let $E_\e:= B(0,R+\e)\setminus B(0,R-\e)$.
By Green's identity, we see that for $x\in E_{\e}$ and $t\in(0,T)$,
\begin{equation*}
\begin{split}
(-\partial_t -\Delta_x) (\varphi_{h,R}*\eta_\e) (x,t)
&= \int_{\mathbb{R}^N} (-\partial_t -\Delta_x)\eta_\e  (x-y)\cdot\varphi_{h,R}(y) \, \di y \\
&= \int_{B(0,R) \cap B(x,\e)} (-\partial_t -\Delta_y)\eta_\e  (x-y)\cdot\varphi_{h,R}(y,t) \, \di y \\
&= \int_{B(0,R) \cap B(x,\e)}\eta_\e  (x-y) (-\partial_t -\Delta_y)\varphi_{h,R}(y,t) \, \di y\\ 
&\qquad+ \int_{\partial B(0,R) \cap B(x,\e)} \eta_\e  (x-y)\frac{\partial  \varphi_{h,R}}{\partial n}(y,t) \, \di \sigma(y)\\
&= \int_{B(0,R) \cap B(x,\e)}\eta_\e  (x-y) h(y,t) \, \di y\\
&\qquad+ \int_{\partial B(0,R) \cap B(x,\e)} \eta_\e  (x-y)\frac{\partial  \varphi_{h,R}}{\partial n}(y,t) \, \di \sigma(y).\\
\end{split}
\end{equation*}
Since $\psi \in L^1_{loc}(\mathbb{R}^N\times[0,{\tau}])$, $\psi \ge 0$ on $\partial B(0,R)\times [0,{\tau}]$, and
\[
\int_{\partial B(0,R) \cap B(x,\e)} \eta_\e  (x-y)\frac{\partial  \varphi_{h,R}}{\partial n}(y,t) \, \di \sigma(y)=
\int_{\partial B(0,R)} \eta_\e  (x-y)\frac{\partial  \varphi_{h,R}}{\partial n}(y,t) \, \di \sigma(y),
\]
by  Hopf's lemma and \eqref{eq:6.6} we obtain
\begin{equation*}
\begin{split}
&\int_{0}^{{\tau}} \int_{E_\e} \psi(-\partial_t -\Delta) (\varphi_{h,R}*\eta_\e)\, \di x\di t\\
&\qquad = \int_{0}^{{\tau}} \int_{E_\e}  \psi(x,t)\int_{B(0,R) }   \eta_\e  (x-y) h(y,t) \, \di y\di x \di t\\
&\qquad\qquad + \int_{0}^{{\tau}} \int_{\mathbb{R}^N}  \psi(x,t)\chi_{E_\e}(x)\int_{\partial B(0,R)}\eta_\e  (x-y) \frac{\partial  \varphi_{h,R}}{\partial n}(y,t) \, \di \sigma(y)\di x\di t\\
&\qquad = \int_{0}^{{\tau}} \int_{E_\e}  \psi(x,t) (h*\eta_\e)(x,t) \,\di x \di t\\
&\qquad\qquad +\int_0^{{\tau}} \int_{\partial B(0,R)} \int_{\mathbb{R}^N} \eta_{\e}(y-x) \psi (x,t) \, \di x \frac{\partial \varphi_{h,R}}{ \partial n}(y,t) \, \di \sigma(y) \di t\\
&\qquad\to  \int_{0}^{{\tau}} \int_{\partial B(0,R)}  \psi(y,t) \frac{\partial  \varphi_{h,R}}{\partial n}(y,t) \, \di \sigma(y)\di t \le 0 \quad \mbox{as} \quad \e\to0.\\
\end{split}
\end{equation*}
Now letting $\e\to0$ in \eqref{eq:NE6}, we have 
\begin{equation*}
\begin{split}
&\int_{0}^{{\tau}} \int_{B(0,R)} \psi(-\partial_t -\Delta) \varphi_{h,R}\, \di x\di t\\
&\qquad=  \int_{B(0,R)} \psi_0 \varphi_{h,R}(0)\, \di x+ \int_{0}^{{\tau}}\int_{B(0,R)} [f(u)- \alpha f( u^*)] \varphi_{h,R} \,\di x \di t \\
&\qquad\qquad-\int_{0}^{{\tau}} \int_{\partial B(0,R)}  \psi(x,t) \frac{\partial  \varphi_{h,R}}{\partial n}(x,t) \, \di \sigma(x)\di t \\
&\qquad\ge \int_{B(0,R)} \psi_0 \varphi_{h,R}(0)\, \di x+ \int_{0}^{{\tau}}\int_{B(0,R)} [f(u)- \alpha f( u^*)] \varphi_{h,R} \,\di x \di t \\
&\qquad= \int_{B(0,R)} \psi_0 \varphi_{h,R}(0)\, \di x+ \int_{0}^{{\tau}}\int_{B(0,R)} [f(\alpha u^*+\psi)- f( \alpha u^*)] \varphi_{h,R} \,\di x \di t \\
&\qquad\qquad+ \int_{0}^{{\tau}}\int_{B(0,R)} [f(\alpha u^*)- \alpha  f( u^*)] \varphi_{h,R} \,\di x \di t \\
&\qquad\ge \int_{B(0,R)} \psi_0 \varphi_{h,R}(0)\, \di x+ \int_{0}^{{\tau}}\int_{B(0,R)} [f(\alpha u^*+\psi)- f( \alpha u^*)] \varphi_{h,R}\, \di x \di t. \\
\end{split}
\end{equation*}
Here, we used
\[
 f(\alpha u^*) -\alpha f(u^*) \ge 0
\]
for $x\in \mathbb{R}^N\setminus\{0\}$, because $f'>0$ and $f''>0$ on $(0, \infty)$, and $\alpha\ge1$.
 Furthermore, since
\begin{equation*}
\begin{split}
&f(\alpha u^*+ \psi) - f(\alpha u^*) \ge f(\alpha u^*) - f(\alpha u^*-\psi),
\end{split}
\end{equation*}
we have
\begin{equation}
\label{eq:NE7}
\begin{split}
&\int_{0}^{{\tau}} \int_{B(0,R)} \psi(-\partial_t -\Delta) \varphi_{h,R}\, \di x\di t  -\int_{B(0,R)} \psi_0 \varphi_{h,R}(0)\, \di x\\
&\qquad=\int_{0}^{{\tau}} \int_{B(0,R)} \psi h\, \di x\di t -\int_{B(0,R)} \psi_0 \varphi_{h,R}(0)\, \di x\\
&\qquad\ge  \int_{0}^{{\tau}}\int_{B(0,R)} [ f(\alpha u^*) - f(\alpha u^*-\psi)] \varphi_{h,R} \, \di x \di t \\
&\qquad\ge \int_{0}^{{\tau}}\int_{B(0,R)} [ f(u^*) - f( u^*-\psi)] \varphi_{h,R} \,\di x \di t. \\
\end{split}
\end{equation}

We  take a nonnegative measurable function $w_0$ in $\mathbb{R}^N$ such that 
\[
0\le w_0  \le u^*_{R}, \quad w_0\not \equiv u^*_{R}, \quad \mbox{and} \quad
u^*_{R} - w_0 \le \psi _0, \quad \mbox{a.e.~in} \quad B(0,R).
\]
For example, taking $w_0 := (u^*_{R} - \psi_0)_+\chi_{B(0,R)}$,
we see  from Proposition~\ref{S6P4} that problem~\eqref{S6E1} with   $w(0) = w_0$
has the maximal solution $w$ in $B(0,R)\times[0,\infty)$ in the sense of Definition \ref{S6D2}.
Then  there exists $\{w^{(k)}\}_{k=0}^\infty$ such that
\begin{equation*}
\begin{split}
&w^{(0)}(x,t) := u^*_{R}(x),\\
&w^{(k+1)} (x,t) := \int_{B(0,R)} G_R(x,y,t) w_0(y) \, \di y + \int_0^t \int_{B(0,R)} G_R(x,y,t-s)f_{R}(w^{(k)}(y,s)) \, \di y\di s, 
\end{split}
\end{equation*}
for $k=0,1,2,\cdots$ and $w^{(k)}$ satisfies
\[
u^*_{R} = w^{(0)} \ge w^{(1)} \ge w^{(2)} \ge \cdots \ge w^{(k)} \ge w^{(k+1)} \ge  \cdots \to w \ge0
\]
as $k\to\infty$. 
If we define $\phi^{(k)} := u^*_{R} - w^{(k)}$, we have
\[
0 = \phi^{(0)} \le \phi^{(1)} \le \phi^{(2)} \le \cdots \le \phi^{(k)} \le \phi^{(k+1)} \le \cdots \to u^*_{R}-w =: \phi < \infty
\]
for a.a.~$B(0,R)\times[0,T)$ as $k\to\infty$.  
Since $f(u^*)\in L^1_{loc}(\mathbb{R}^N\times[0,{\tau}])$ and $f_{R}(u^*_{R})\in L^1_{loc}(B(0,R)\times[0,{\tau}])$, by a similar computation to Theorem~\ref{S7T1} below, $\{w^{(k)}\}_{k=0}^\infty$ and $\{\phi^{(k)}\}_{k=0}^\infty$  satisfy
\begin{equation*}
\begin{split}
&\int_{0}^{{\tau}} \int_{B(0,R)} w^{(k+1)}(-\partial_t -\Delta)\varphi_{h,R}\, \di x\di t \\
&\qquad= \int_{0}^{{\tau}} \int_{B(0,R)} w^{(k+1)}h\, \di x\di t \\
&\qquad= \int_{B(0,R)} w_0 \varphi_{h,R}(0)\, \di x
+ \int_{0}^{\tau} \int_{B(0,R)} f_{R}(w^{(k)}) \varphi_{h,R}\, \di x \di t 
\end{split}
\end{equation*}
and
\begin{equation}
\label{eq:NE8}
\begin{split}
&\int_{0}^{{\tau}} \int_{B(0,R)} \phi^{(k+1)}(-\partial_t -\Delta) \varphi_{h,R}\, \di x\di t  -\int_{B(0,R)} \phi_0 \varphi_{h,R}(0)\, \di x\\
&\qquad=\int_{0}^{{\tau}} \int_{B(0,R)} \phi^{(k+1)}h\, \di x\di t -\int_{B(0,R)} \phi_0 \varphi_{h,R}(0)\, \di x\\
&\qquad=\int_{0}^{{\tau}}\int_{B(0,R)} [f_{R}(u^*_{R}) -f_{R}(w^{(k)}) ]\varphi_{h,R}\di x \di t \\
&\qquad=\int_{0}^{{\tau}}\int_{B(0,R)} [f(u^*) -f(w^{(k)} + u^*(R)) ]\varphi_{h,R} \di x \di t\\
&\qquad= \int_{0}^{{\tau}}\int_{B(0,R)} [f(u^*) + f(u^*-\phi^{(k)})) ]\varphi_{h,R} \di x \di t\\
\end{split}
\end{equation}
for  any $0<{\tau}<T$ and for $k=1,2,\cdots$, where $\phi_0 := u^*_{R}-w_0$.
Note that $\phi^{(k)}=0$ on $\partial B(0,R)$ for all $k=0,1,2,\cdots$.

We claim that $\phi\le \psi $ for a.a.~$(x,t)\in B(0,R)\times (0,T)$.
It suffices to show that $\phi^{(k)} \le \psi$ for all $k=0,1,2,\cdots$.
It is obvious that $0= \phi^{(0)} \le \psi$. We assume that $\phi^{(k)} \le \psi$ for some $k\ge0$.
By \eqref{eq:NE7} and \eqref{eq:NE8} we have
\begin{equation*}
%\label{eq:NE5}
\begin{split}
&\int_{0}^{{\tau}} \int_{B(0,R)} (\psi-\phi^{(k+1)})h\, \di x\di t\\
&\qquad\ge \int_{B(0,R)} (\psi_0-\phi_0) \varphi(0)\, \di x+ \int_{0}^{{\tau}}\int_{B(0,R)} [  f(u^*-\phi^{(k)})  -  f(u^*-\psi)] \varphi_{h,R}\,  \di x \di t \\
&\qquad\ge0
\end{split}
\end{equation*}
for  any $0<{\tau}<T$.
Since $0<{\tau}<T$ and $h$ are arbitrary, we see that $\phi^{(k+1)} \le \psi$ for a.a.~$(x,t)\in B(0,R)\times (0,T)$.
Then the claim follows and we obtain 
\[
\phi(x,t) = u^*(x)- u^*(R) - w(x,t) \le u(x,t) - \alpha u^*(x)  = \psi (x,t)
\]
for a.a.~$(x,t)\in B(0,R)\times (0,T)$. 
Since $\lim_{r\to 0} u^*(r) = \infty$ and $w\in L^{\infty}_{loc} (0,\infty; L^\infty(B(0,R)))$,
for any $t_2\in(0,T)$,
there exists $R'<R$ such that 
\[
\left(\alpha +\frac{1}{2} \right)u^*(x) < (\alpha + 1)u^*(x) -u^*(R)-  \esssup\displaylimits_{\tau\in (t_2, T)}\left\|w(\tau)\right\|_{L^{\infty}(B(0,R))} \le u(x,t)
\]
for a.a.~$(x,t)\in B(0,R')\times (t_2,T)$. 
Thus, the proof is complete.
\end{proof}

\begin{proof}[Proof of Theorem~{\rm \ref{S1T1} (iii)}.]
Let $u_0$ be as in Theorem~\ref{S1T1} (iii).
Assume that problem~\eqref{S1E1} with $u(0)=u_0$ has a solution $u$ in $\mathbb{R}^N\times[0,T)$ in the sense of Definition~\ref{S1D1} such that $u(x,t)\ge u^*(x)$, where $T\in (0,\infty)$.
We can assume without loss of generality that $T>0$ is sufficiently small.
Define $\psi(x,t) := u(x,t) - u^*(x)$.
Note that $\psi\ge 0$ by the assumption.
It follows from Theorem~\ref{S7T1} that 
\begin{equation*}
\begin{split}
&\int_{0}^{{\tau}} \int_{\mathbb{R}^N} \psi(-\partial_t -\Delta) \varphi \, \di x\di t\\
&\qquad= \int_{\mathbb{R}^N} \psi_0 \varphi(0)\, \di x+ \int_{0}^{{\tau}} \int_{\mathbb{R}^N} [f(u^*+\psi)- f( u^*)] \varphi  \di x \di t 
\end{split}
\end{equation*}
for any $0<{\tau}<T$ and any $0\le \varphi \in   C^{2,1}_0 (\mathbb{R}^N\times[0,{\tau}])$ with $\varphi({\tau})=0$,
where $\psi_0 = u_0 -  u^*  \ge0$.
Since $f'>0$ and $f''>0$ on $(0, \infty)$, we have
\[
f(u^*) - f(u^*-\psi) \ge0, \qquad f(u^*+\psi) - f(u^*) \ge0,
\]
and
\[
f(u^*+ \psi) - f(u^*) \ge f(u^*) - f(u^*-\psi).
\]
The above equality yields
\begin{equation}
\label{eq:4.1}
\begin{split}
&\int_{0}^{{\tau}} \int_{\mathbb{R}^N} \psi(-\partial_t -\Delta) \varphi \, \di x\di t\\
&\qquad\ge \int_{\mathbb{R}^N}\psi_0 \varphi(0)\, \di x+ \int_{0}^{{\tau}}\int_{\mathbb{R}^N} [f(u^*) - f(u^*-\psi)] \varphi \, \di x \di t.
\end{split}
\end{equation}

Now we choose a nonnegative function $w_0$ on $\mathbb{R}^N$ such that
$w_0\le u^*$ a.e.~in $\mathbb{R}^N$, $w_0 \not\equiv u^*$ a.e.~in $\mathbb{R}^N$, and
\begin{equation}
\label{eq:4.2}
0 \le u^*(x) - w_0(x) \le u_0(x) - u^*(x) = \psi_0 \quad \mbox{for a.a.}\ x\in \mathbb{R}^N.
\end{equation}
Then there exist the maximal solution $w$ of problem~\eqref{S1E1} with $u(0) = w_0$
and  $\{w^{(k)}\}_{k=0}^\infty$ such that
\begin{equation*}
\begin{split}
&w^{(0)}(t) := u^*,\\
&w^{(k+1)} (t) := S(t) w_0 + \int_0^t S(t-s)f(w^{(k)}(s)) \, \di s \quad \mbox{for} \quad k=0,1,2,\cdots,
\end{split}
\end{equation*}
and $w^{(k)}$ satisfies
\[
u^* = w^{(0)} \ge w^{(1)} \ge w^{(2)} \ge \cdots \ge w^{(k)} \ge w^{(k+1)} \ge  \cdots \to w \ge0
\]
as $k\to\infty$. 
In particular, $w\in L^\infty_{loc} (0,\infty; L^\infty(\mathbb{R}^N))$.
Let $\phi^{(k)} := u^* - w^{(k)}$. Then,
\[
0 = \phi^{(0)} \le \phi^{(1)} \le \phi^{(2)} \le \cdots \le \phi^{(k)} \le \phi^{(k+1)} \le \cdots \to u^*-w =: \phi < \infty
\]
for a.a.~$(x,t)\in\mathbb{R}^N\times[0,\infty)$ as $k\to\infty$.  
Since $f(u^*)\in L^1_{loc}(\mathbb{R}^N)$, by a similar computation to Theorem~\ref{S7T1} below,  $\{w^{(k)}\}_{k=0}^\infty$ and $\{\phi^{(k)}\}_{k=0}^\infty$  satisfy
\begin{equation*}
\int_{0}^{{\tau}} \int_{\mathbb{R}^N} w^{(k+1)}(-\partial_t -\Delta) \varphi \, \di x\di t = \int_{0}^{{\tau}} \int_{\mathbb{R}^N} f(w^{(k)}) \varphi  \di x \di t + \int_{\mathbb{R}^N} w_0 \varphi(0)\, \di x
\end{equation*}
and
\begin{equation*}
\begin{split}
&\int_{0}^{{\tau}} \int_{\mathbb{R}^N} \phi^{(k+1)}(-\partial_t -\Delta) \varphi \, \di x\di t \\
&\qquad= \int_{0}^{{\tau}}\int_{\mathbb{R}^N} [f(u^*) -f(w^{(k)}) ]\varphi \, \di x \di t + \int_{\mathbb{R}^N} \phi_0 \varphi(0)\, \di x\\
&\qquad= \int_{0}^{{\tau}}\int_{\mathbb{R}^N} [f(u^*) -f(u^*-\phi^{(k)}) ]\varphi  \,\di x \di t + \int_{\mathbb{R}^N} \phi_0 \varphi(0)\, \di x
\end{split}
\end{equation*}
for  any $0<{\tau}<T$ and any $0\le \varphi \in C^{2,1}_0 (\mathbb{R}^N\times[0,{\tau}])$ with $\varphi({\tau})=0$
for $k=1,2,\cdots$, where $\phi_0 := u^*-w_0$.

We claim that $\phi \le \psi$ for a.a.~$(x,t)\in \mathbb{R}^N\times (0,T)$.  It suffices to show $\phi^{(k)} \le \psi$ for all $k=0,1,2,\cdots$.
It is obvious that $0= \phi^{(0)} \le \psi$. We assume that $\phi^{(k)} \le \psi$ for some $k\ge0$.
By \eqref{eq:4.1} and \eqref{eq:4.2} we have
\begin{equation}\label{CP4H}
\begin{split}
\infty&>\int_{0}^{{\tau}} \int_{\mathbb{R}^N} (\psi-\phi^{(k+1)})(-\partial_t -\Delta) \varphi \, \di x\di t\\
&\ge \int_{\mathbb{R}^N} (\psi_0-\phi_0) \varphi(0)\, \di x+ \int_{0}^{{\tau}}\int_{\mathbb{R}^N} [ f(u^*-\phi^{(k)})  - f(u^*-\psi) ] \varphi  \di x \di t \\
&\ge  0\\
\end{split}
\end{equation}
for any $0<{\tau}<T$ and  any $0\le \varphi \in  C^{2,1}_0 (\mathbb{R}^N\times[0,{\tau}])$ with $\varphi({\tau})=0$.
As in the proof of Lemma~\ref{S6L5} we shall employ $\varphi_{h,R}*\eta_{\e}$ given by \eqref{eq:NE6-} as a test function.
Then we take $\e\to 0$ to obtain \eqref{CP4H} with $\varphi_{h,R}$ instead of $\varphi$.
Since $h$, $R$, and $\tau$ are arbitrary, we see that $\phi^{(k+1)} \le \psi$ for a.a.~$(x,t)\in \mathbb{R}^N\times (0,T)$.
Then the claim follows and we have 
\[
\phi(x,t) = u^*(x) - w(x,t) \le u(x,t) - u^*(x)  = \psi (x,t)
\]
for a.a.~$(x,t)\in \mathbb{R}^N\times (0,T)$.
Since $\lim_{r\to 0} u^*(r) = \infty$ and $w\in L^{\infty}_{loc} (0,\infty; L^\infty(\mathbb{R}^N))$,
for any $t_1 \in (0,T)$, there exists $R>0$ such that 
\[
\frac{3}{2}u^*(x) <  2u^*(x) -  \esssup\displaylimits_{\tau\in (t_1, T)}\left\|w(\tau)\right\|_{L^\infty (\R^N)} \le u(x,t)
\]
for a.a.~$(x,t)\in B(0,R)\times (t_1,T)$.
We can arbitrarily choose $t_1,t_2 \in (0,T)$ in Lemma~\ref{S6L5}.
Applying Lemma~\ref{S6L5} finitely many times, we see that
for any sufficiently large $\alpha^*>1$ and any $\tau^* \in (0,T)$, there exists  $R^*>0$ 
such that
\begin{equation}\label{eq:4.4}
\alpha^* u^*(x)  \le  u(x,t) 
\end{equation}
for a.a.~$(x,t)\in B(0,R^*)\times (\tau^*,T)$.
Since $\log f(u)$ is convex for sufficiently large $u>0$, we have
\[
\log f(\alpha^* u) \ge \alpha^* \log f(u) \quad \mbox{for sufficiently large} \quad u>0,
\]
and this implies that
\begin{equation*}
%\label{eq:4.6}
f(\alpha^* u) \ge f(u)^{\alpha^*}\quad \mbox{for sufficiently large} \quad u>0.
\end{equation*}
If we take  a sufficiently large $\alpha^*>1$ and  a sufficiently small $r^*>0$ if necessary, this together with  Lemma~\ref{S3L4} and \eqref{eq:4.4} implies that
\begin{equation*}
\begin{split}
\int_{B(0,r^*)} f(u(x,t)) \, \di x 
&\ge \int_{B(0,r^*)} f(\alpha^* u^*(x)) \, \di x \\
&\ge \int_{B(0,r^*)} f( u^*(x))^{\alpha^*} \, \di x \\
&\ge C \int_0^{r^*} f( u^*(r))^{\alpha^*} r^{N-1} \, \di r \\
&\ge C  \int_0^{r^*}r^{N-2\alpha^*(1-\delta)-1}  \, \di r =\infty \\
\end{split}
\end{equation*}
for a.a.~$t\in (\tau^*, T)$.
This implies that $f(u) \not\in L^1_{\loc}(\mathbb{R}^N\times[0,T))$.
Therefore, we see that problem~\eqref{S1E1} has no local-in-time solutions.
Thus, the proof is complete. 
\end{proof}

%%%%%%%%%%%%%%%%%%%%%%%%%%%%%%%%%%%%%%%%%%%%%%%%%%%%%%%
%%%%%%%%%%%%%%%%%%%%%%%%%%%%%%%%%%%%%%%%%%%%%%%%%%%%%%%
%%%%%%%%%%%%%%%%%%%%%%%%%%%%%%%%%%%%%%%%%%%%%%%%%%%%%%%
% Section 7
%%%%%%%%%%%%%%%%%%%%%%%%%%%%%%%%%%%%%%%%%%%%%%%%%%%%%%%
%%%%%%%%%%%%%%%%%%%%%%%%%%%%%%%%%%%%%%%%%%%%%%%%%%%%%%%
%%%%%%%%%%%%%%%%%%%%%%%%%%%%%%%%%%%%%%%%%%%%%%%%%%%%%%%
%\newpage
\section{\bf Appendix}
We study a relation between two notions of solutions.
\begin{theorem}\label{S7T1}
If $u$ is a solution of problem~\eqref{S1E1} in the sense of Definition~{\rm \ref{S1D1}}, then $u$ is also a weak solution of problem~\eqref{S1E1} in the sense of Definition~{\rm \ref{S6D1}}.
\end{theorem}
Because of Theorem~\ref{S7T1}, a solution of \eqref{S1E1} given by Theorem~\ref{S1T1}~(i) is also a weak solution.
\begin{proof}
Let $T\in (0,\infty]$ and $\e \in (0,T/2)$.
We find $t\in(T-\e,T)$ such that
\begin{equation*}
\begin{split}
\infty>u(x,t) &\ge \int_{0}^{T-2\e} \int_{\mathbb{R}^N} G(x-y,t-s) f(u(y,s)) \,\di y \di s\\
&=\frac{1}{(4\pi)^\frac{N}{2}} \int_0^{T-2\e} \int_{\mathbb{R}^N}  (t-s)^{-\frac{N}{2}} \exp\left(-\frac{  |x-y|^2}{4(t-s)}\right) f(u(y,s))\, \di y\di s\\
&\ge  CT^{-\frac{N}{2}} \int_0^{T-2\e} \int_{\mathbb{R}^N}  \exp\left(-\frac{ C |x-y|^2}{\e}\right) f(u(y,s))\,\di y \di s
\end{split}
\end{equation*}
for a.a.~$x\in\mathbb{R}^N$. Since $\e\in(0,T/2)$ is arbitrary, we see that $f(u) \in L^1_{loc} (\mathbb{R}^N\times [0,T))$.
For any $0<{\tau}<T$, let $\varphi\in C^\infty_0(\mathbb{R}^N \times[0,{\tau}])$ with $\varphi({\tau})=0$. 
By the integral by parts  we have
\begin{equation*}
\begin{split}
&\int_{\mathbb{R}^N} u_0(y)\varphi(y,0) \, \di y\\
&\qquad= \int_{\mathbb{R}^N} \left(\int_0^{{\tau}}\int_{\mathbb{R}^N}(\partial_t-\Delta_x)G(x-y,t)\cdot \varphi(x,t)\,\di x \di t+ \varphi(y,0)\right) u_0(y)\,\di y\\
&\qquad= \int_{\mathbb{R}^N} \int_0^{{\tau}}\int_{\mathbb{R}^N}G(x-y,t)(-\partial_t-\Delta_x)\varphi(x,t) \cdot  u_0(y) \,\di x \di t \di y\\
&\qquad= \int_0^{{\tau}}\int_{\mathbb{R}^N}\left(\int_{\mathbb{R}^N}G(x-y,t)u_0(y)\,\di y\right)(-\partial_t-\Delta_x)\varphi(x,t)\,\di x \di t. 
\end{split}
\end{equation*}
Similarly, we have
\begin{equation*}
\begin{split}
&\int_0^{{\tau}}\int_{\mathbb{R}^N} \varphi(y,s)  f(u(y,s))\, \di y \di s\\
&\qquad= \int_0^{{\tau}}\int_{\mathbb{R}^N} \left(\int_s^{{\tau}}\int_{\mathbb{R}^N}(\partial_t-\Delta_x)G(x-y,t-s)\cdot \varphi(x,t)\,\di x\di t + \varphi(y,s)\right.\\
&\qquad\qquad\times f(u(y,s)\bigg)\,\di y\di s\\
&\qquad= \int_0^{{\tau}}\int_{\mathbb{R}^N} \left(\int_s^{{\tau}}\int_{\mathbb{R}^N}G(x-y,t-s)(-\partial_t-\Delta_x)\varphi(x,t)\,\di x\di t \right)f(u(y,s))\,\di y\di s\\
&\qquad= \int_0^{{\tau}}\int_{\mathbb{R}^N} \left(\int_0^{{\tau}}\int_{\mathbb{R}^N}G(x-y,t-s )f(u(y,s))\,\di y \di s \right)(-\partial_t-\Delta_x)\varphi(x,t)\,\di x \di t.
\end{split}
\end{equation*}
Then
\begin{equation*}
\begin{split}
&\int_{0}^{{\tau}}\int_{\mathbb{R}^N} u(x,t)(-\partial_t-\Delta_x) \varphi(x,t)\,\di x\di t\\
&\qquad= \int_{0}^{{\tau}}\int_{\mathbb{R}^N} \biggl(\int_{\mathbb{R}^N}G(x-y,t) u_0(y)\, \di y\\
&\qquad\qquad+\int_0^{{\tau}}\int_{\mathbb{R}^N}G(x-y, t-s)f(u(y,s))\,\di y\di s\biggr)(-\partial_t-\Delta_x) \varphi(x,t)\,\di x\di t\\
&\qquad= \int_{\mathbb{R}^N} \varphi(y,0) u_0(y)\, \di y+ \int_0^{{\tau}}\int_{\mathbb{R}^N} \varphi(y,s) f(u(y,s))\, \di y\di s, \end{split}
\end{equation*}
which implies \eqref{S6D1E1}. 
Then Theorem~\ref{S7T1} follows.
\end{proof}

Combining Theorem~\ref{S1T1}~(i) and (ii) and Theorem~\ref{S7T1}, we obtain the following:
\begin{corollary}
The following hold:\\
{\rm (i)} Let $u$ be a nonnegative solution of problem~\eqref{S1E1} obtained in Theorem~{\rm \ref{S1T1}~(i)} in the sense of Definition~{\rm \ref{S1D1}}.
Then $u$ is also a weak solution in the sense of Definition~{\rm \ref{S6D1}}.\\ 
{\rm (ii)} Let $u^*$ be the singular stationary solution of problem~\eqref{S1E1} obtained in Theorem~{\rm \ref{S1T1}~(ii)} in the sense of Definition~{\rm \ref{S1D1}}.
Then $u^*$ is also a weak solution in the sense of Definition~{\rm \ref{S6D1}}.\\
\end{corollary}

\noindent
{\bf Acknowledgement}\\
The authors would like to thank the referee for pointing out a serious gap in the proof of Theorem~\ref{S1T1}~(iii) of the previous version.
Without his/her comments this research does not accomplish.
YM was supported by JSPS KAKENHI Grant Number 24K00530.\\

\noindent{\bf Conflict of interest}\\
The authors have no relevant financial or non-financial interests to disclose.\\

\noindent{\bf Data Availability}\\
Data sharing is not applicable to this article as no new data were created or analyzed in this study.

%%%%%%%%%%%%%%%%%%%%%%%%%%%%%%%%%%%%%%%%%%%%%%%%%%%%%%%
%%%%%%%%%%%%%%%%%%%%%%%%%%%%%%%%%%%%%%%%%%%%%%%%%%%%%%%
%%%%%%%%%%%%%%%%%%%%%%%%%%%%%%%%%%%%%%%%%%%%%%%%%%%%%%%
% Bibliography
%%%%%%%%%%%%%%%%%%%%%%%%%%%%%%%%%%%%%%%%%%%%%%%%%%%%%%%
%%%%%%%%%%%%%%%%%%%%%%%%%%%%%%%%%%%%%%%%%%%%%%%%%%%%%%%
%%%%%%%%%%%%%%%%%%%%%%%%%%%%%%%%%%%%%%%%%%%%%%%%%%%%%%%

\end{document}